\documentclass[a4paper,10pt, leqno]{amsart}
\usepackage[active]{srcltx}
\usepackage{graphicx}
\usepackage{amssymb}
\usepackage[usenames]{color}

\numberwithin{equation}{section}
\newtheorem{Theorem}{Theorem}[section]

\newtheorem{Fact}[Theorem]{Fact}
\newtheorem{Corollary}[Theorem]{Corollary}
\newtheorem{Cor}[Theorem]{Corollary}
 \newtheorem{Lemma}[Theorem]{Lemma}
\newtheorem{Proposition}[Theorem]{Proposition}
\newtheorem{Prop}[Theorem]{Proposition}
\theoremstyle{remark}
\newtheorem{Remark}[Theorem]{Remark}
\newtheorem{Rmk}[Theorem]{Remark}
\theoremstyle{definition}
\newtheorem{Definition}[Theorem]{Definition}

\newtheorem{Example}[Theorem]{Example}

\newtheorem*{acknowledgements}{Acknowledgements}
\newcommand{\vect}[1]{\mathbf #1}
\newcommand{\inner}[2]{\left\langle{#1},{#2}\right\rangle}

\newcommand{\R}{{\mathbb R}}

\newcommand{\mc}[1]{{\mathcal #1}}
\newcommand{\mb}[1]{{\mathbf #1}}
\newcommand{\pmt}[1]{{\begin{pmatrix} #1  \end{pmatrix}}}

\renewcommand{\phi}{\varphi}
\renewcommand{\epsilon}{\varepsilon}
\newcommand{\op}[1]{{\operatorname{ #1}}}
\newcommand{\dy}{\displaystyle}

\renewcommand{\phi}{\varphi}
\renewcommand{\epsilon}{\varepsilon}

\renewcommand{\phi}{\varphi}

\title{Deformations of swallowtails \\
in a 3-dimensional space form}

\author{K.~Saji}
\address[Kentaro Sajia]{
  Department of Mathematics,
  Faculty of Science,
  Kobe University,
  Rokko, Kobe 657-8501}
\email{saji@math.kobe-u.ac.jp}

\author{M.~Umehara}
\address[Masaaki Umehara]{%
   Department of Mathematical and Computing Sciences,
   Institute of Science Tokyo,
   2-12-1-W8-34, O-okayama, Meguro-ku,
   Tokyo 152-8552, Japan.
}
\email{umehara@is.titech.ac.jp}

\author{K.~Yamada}
\address[Kotaro Yamada]{%
   Department of Mathematics,
   Institute of Science Tokyo,
   O-okayama, Meguro, Tokyo 152-8551, 
   Japan
}
\email{kotaro@math.titech.ac.jp}

\date{October 21, 2024}
\thanks{The first author was supported in part 
by Grant-in-Aid for Scientific Research (C)
No. 22K03312. The second and the third authors were
supported in part 
by Grant-in-Aid for Scientific Research
(B) No.23K20794 and (B) No.23K22392 respectively,
from Japan Society for the Promotion of Science. }
\begin{document}

\maketitle
\begin{abstract}
This is a continuation of the authors' earlier work on 
deformations of cuspidal edges. 
We give a representation formula for  
swallowtails in the Euclidean 3-space.
Using this, we investigate map germs of generic swallowtails 
in 3-dimensional space from,  and show some important properties 
of them. In particular,
we give a representation formula giving all map germs 
of swallowtails in the Euclidean 3-space
whose Gaussian curvatures 
are bounded from below by a positive constant or
by a negative constant from above.
Using this, we show that any 
swallowtails are deformed into a
swallowtail of constant Gaussian curvature preserving
the sign of their Gaussian curvatures.
\end{abstract}

\section*{Introduction}
We fix an oriented Reimannian $3$-manifold $(M^3,g)$.
Let $\mathcal F(M^3)$ be the set of 
germs of $C^\infty$-maps 
into $M^3$ 
at the origin 
$$
o:=(0,0)\in \R^2.
$$ 
We suppose that $f_0,f_1\in \mathcal F(M^3)$.
As we defined in the authors' previous work \cite{SUY2022}, 
we say that {\it $f_0$ is
possible to deform  to $f_1$ in $\mathcal F(M^3)$}
if there exists a neighborhood $U(\subset \R^2)$ of $o$
such that $f_0$ and $f_1$ are both defined on $U$
and there exists a $C^\infty$-map
$
F:[0,1]\times U\to M^3
$
such that
$$
F(0,x)=f_0(x),\qquad F(1,x)=f_1(x) \qquad (x\in U)
$$
and each
map-germ $x\mapsto F(t,x)$ ($t\in [0,1]$)
belongs to $\mathcal F(M^3)$.

In \cite{SUY}, the authors investigated 
deformations of cuspidal edges.
The purpose of this paper is to investigate deformations of
``swallowtails":
A germ $f$ in $\mathcal F(M^3)$ is called a  {\it swallowtail}
if there exists a diffeomorphism germ $\phi$ of $\R^2$ at the origin
and a diffeomorphism germ $\Phi$ of $M^3$ to $\R^3$
satisfying $\Phi(f(p))=\mb 0$
($\mb 0:=(0,0,0)$)
such that $\Phi\circ f\circ \phi$
coincides with the standard swallowtail
\begin{equation}\label{eq:fs173}
f_S(u,v):=(u,2v^3+uv,3v^4+uv^2).
\end{equation}
We fix a $C^\infty$-map $f:U\to (M^3,g)$
giving a swallowtail singular point at $o$, and set 
$$
f_u:=df(\partial/\partial u),\qquad f_v:=df(\partial/\partial v).
$$
We denote by $\Sigma_f(\subset U)$ the singular set of $f$.

\begin{Definition}\label{def:187}
Let $f(u,v)$  be an element of
$\mathcal F(M^3)$ giving 
a swallowtail singular point at the origin $o$.
Then the parametrization of $f$ is said to be {\it admissible} if
the $u$-axis is the singular set.
\end{Definition}

An arbitrarily given germ of
a swallowtail has an admissible parametrization
by a suitable coordinate change of the domain of definition.
For example, 
$f_S(-6\xi^2-\eta,\xi)$
gives a swallowtail whose singular set is the $\xi$-axis.
Consider another parametrization $(\xi,\eta)$ centered at 
the origin $o\in U$.
By definition, the coordinate system $(\xi,\eta)$ is
an admissible parametrization of $f$.

We denote by $\mathcal F^S(M^3)$  the set of 
germs of swallowtails at the origin $o\in \R^2$ into $M^3$
parametrized by an admissible parametrization.
In authors' previous work \cite{MSUY} with Martins,
the limiting normal curvature $\kappa_\nu$
and the normalized cuspidal curvature $\mu_c$
are defined at swallowtail singular points.
We let $\sigma^S_0$ and $\sigma^S_g$ be
the signs of the numerator of these two invariants, respectively
(see Section~1 for details).

\begin{Definition}\label{def:227}
A germ of swallowtail $f:U\to M^3$ is called {\it generic}
if $\sigma_g^S$ does not vanish (see also 
Definition~\ref{def:C298}).
\end{Definition}

We denote by $\mc F_*^S(M^3)$  
the set of germs of generic swallowtails in $M^3$
parametrized by admissible parametrizations.
By definition, 
we have $\mc F_*^S(M^3)\subsetneq \mc F^S(M^3)$.

We let $M^3(a)$ be the space form of constant 
curvature $a\in \R$, that is,
the complete simply connected Riemannian $3$-manifold of constant
sectional curvature $a$.
We first give a representation formula
which produces arbitrary 
germs of swallowtails in the Euclidean 3-space
$M^3(0)(=\R^3)$ (see Theorem~\ref{prop:SW-rep}).
Using this, we show the following results for
generic swallowtails:

\medskip
\noindent
{\bf Theorem~A.}
{\it \it Let $f_i$ $ (i=1,2)$ be two germs of generic swallowtails 
in $\mc F_*^S(M^3(a))$
with common $\sigma^S_0$ and $\sigma^S_g$.
Then $f_1$ can be deformed to $f_2$ in $\mc F_*^S(M^3(a))$.
}

\medskip
\noindent
{\bf Proposition B.} \label{prop:B}
{\it For any space-cusp $\gamma$ 
$($cf. Definition~\ref{def:gen348}$)$ 
in $M^3(a)$,
there exists a generic swallowtail along $\gamma$.}

\medskip
The corresponding assertion for Theorem~A
for cuspidal edges were shown in
\cite[Theorem A]{SUY2022}.
We next consider non-generic case.

\begin{Definition}\label{def:asymptotic}
A germ of swallowtail in $M^3(a)$
whose limiting normal curvature vanishes identically
along the singular set
is said to be {\it asymptotic}.
\end{Definition}

In \cite[Theorems~3.9 and 4.4]{MSUY}, 
the following assertion was shown:

\begin{Fact}\label{fact:I}
{\it For $f\in \mc F(M^3)$,
the following two conditions are equivalent;
\begin{enumerate}
\item $f$ is asymptotic,
\item the Gaussian curvature function $K$ of $f$ is bounded near $o$.
\end{enumerate}
Moreover, in such a case, the Gaussian curvature $K$
and the extrinsic Gaussian curvature $K_{ext}$
can be uniquely extended as smooth  functions}
on a neighborhood of the singular point 
$(K$ coincides with $a+K_{ext}$ if $M^3:=M^3(a))$.
\end{Fact}

Frontals are smooth maps with unit normal vector fields,
and wave fronts are defined as a special class of frontals.
In this paper, we define \lq\lq generalized swallowtails"
(cf. Definition \ref{def:GS})
as a subclass of frontals.
A germ of generalized swallowtail gives a swallowtail if and only 
if it is a wave front.
We denote by $\mc F^{GS}(M^3)$ the set of germs of generalized
swallowtails at $o$.
Our two invariants $\sigma^S_0$ and $\sigma^S_g$
are both extended on $\mc F^{GS}(M^3)$
so that $\sigma^S_0\ne 0$ holds if and only if  $o$ is a swallowtail 
singular point (cf. Fact \ref{fact544}).
We give the following:

\begin{Definition}\label{def:311}
A germ of generalized swallowtail $f:U\to M^3$ is called {\it generic}
if $\sigma_g^S$ does not vanish.
\end{Definition}

It is interesting that there is a 
characterization of non-generic swallowtails:

\medskip
\noindent
{\bf Theorem C.} \label{Thm:S}
{\it Let $f:U\to M^3(a)$ be a 
non-generic $($i.e.~$\sigma^S_g=0)$ 
generalized swallowtail
in $\mc F^{GS}(M^3(a))$.
Then the following two conditions are equivalent:
\begin{enumerate}
\item $o$ is a swallowtail singular point,
\item the singular set image of $f$
is a generic space-cusp in $M^3(a)$.
\end{enumerate}}

\medskip
As an application, we can prove the following
(the first part of the corollary is immediate from 
Theorem C):

\medskip
\noindent
{\bf Corollary C${}'$.} 
{\it The singular set image of a germ of 
non-generic swallowtail  
must be a generic space-cusp 
in $M^3(a)$.
In particular, there are
no asymptotic swallowtails along a non-generic space-cusp.
Moreover, for any given generic space-cusp 
$\gamma$ in $M^3(a)$, there exists
an asymptotic swallowtail along $\gamma$.}

\medskip
This corollary implies that
the singular set image of an asymptotic 
swallowtail in the Euclidean 3-space cannot be planar, although 
there exists a generic swallowtail whose singular set lies in a plane
(see Example \ref{eq:nonASs}).
We also remark that the assertion as in Corollary~C${}'$ is not expected for cuspidal 
edges in general. In fact, there are rotationally symmetric surfaces 
(in the Euclidean 3-space) of positive (resp. negative)
Gaussian curvature  whose singular set image lies in a plane 
(see \cite[Appendices B.7 and B.9]{UY}).

\medskip
We denote by $\mathcal F_0^S(M^3(a))$ the set of 
germs of asymptotic swallowtails
at the origin $o\in \R^2$ into $M^3(a)$
parametrized by admissible parametrization.

\begin{Definition}\label{def:369}
A germ of a swallowtail satisfying $K_{ext}$ is positive (resp. negative)
at the singular point
is said to be {\it positively curved}
(resp. {\it negatively curved}).
\end{Definition}

We give a formula producing all of 
positively or negatively curved swallowtails (see Theorem~\ref{thm:main}) and 
prove the following:

\medskip
\noindent
{\bf Theorem D.} \label{thm:C}
{\it Let $f_i$ $ (i=1,2)$ be germs of asymptotic swallowtails 
in $\mathcal F_0^S(M^3(a))$
with common $\sigma^S_0$.
Then $f_1$ can be deformed to $f_2$ in $\mc F_0^S(M^3(a))$.
Moreover, if $f_i$ $ (i=1,2)$ are both
positively $($resp. negatively$)$ curved,
then the deformation preserves the sign of $K_{ext}$.
}

\medskip
Like as the case of Theorem~A, the corresponding assertion for 
cuspidal edges were shown in the authors' previous work \cite{SUY2022}.
As a corollary, the following assertion holds:

\medskip
\noindent
{\bf Corollary D${}'$.}
{\it Any positively $($resp. negatively$)$ curved asymptotic 
swallowtails in $\mathcal F_0^S(M^3(a))$ 
can be deformed to a swallowtail of constant extrinsic Gaussian curvature
preserving the sign of
$K_{ext}$.}

\section{Preliminaries}

\subsection*{Definition of two signs for swallowtails}
Let $(M^3,g)$ be a Riemannian 3-manifold.
 Using the Riemannian metric $g$ of $M^3$,
 we define the vector product $\times_{g}$ of
 tangent 
 vectors $\vect{a}$, $\vect{b}\in T_PM^3$ at $P\in M^3$ 
by
 \[
   g_P^{}(\vect{a}\times_{g} \vect{b},\vect{c})
    =\op{det}_g(\vect{a}, \vect{b}, \vect{c})
    \qquad (\vect{c}\in T_PM^3).
 \]
When $(M^3,g)$ is the Euclidean 3-space, the product $\times_g$ coincides
with the canonical vector product $\times$.
 We also set $|\vect{a}|:=\sqrt{g_P(\vect{a},\vect{a})}$.  
We denote by $\nabla$ the Levi-Civita connection of $(M^3,g)$, 
and set
\begin{equation}\label{eq:266}
 \nabla_u f_u:=\nabla_{\partial/\partial u}f_u,\quad 
  \nabla_v f_u:=\nabla_{\partial/\partial v}f_u,\quad 
  \nabla_v f_v:=\nabla_{\partial/\partial v}f_v
\end{equation}
along $f$. Since $\nabla$ is of torsion free,
we have that
$
\nabla_v f_u=\nabla_u f_v. 
$
For the sake of simplicity, we will use the notation
\begin{equation}\label{eq:422}
\inner{\vect{a}}{\vect{b}}:=
g_P^{}(\vect{a},\vect{b}) \qquad (\vect{a}, \vect{b}\in T_PM^3,\,\,P\in M^3).
\end{equation}
We let $U$ be an open neighborhood of 
the origin $o\in \R^2$. 

\begin{Definition}
A $C^\infty$-map and $f:U\to M^3$ is called 
a {\it frontal} (at $o$)
if there exists a unit normal vector field $\nu$ of $f$
defined on $U$. We denote by $T_1M^3$ the unit tangent bundle 
of $(M^3,g)$. If $\nu:U\to T_1M^3$ gives an immersion at 
$o$, we call $f$ is a {\it wave front} (at $o$).
\end{Definition}

We assume that $f$ is a frontal at $o$.
Then a point $p\in U$ is  a 
singular point of $f$
if it is the zero of the function
$\lambda:=\det_g(f_u,f_v,\nu)$ given in \eqref{eq:lambda}.

\begin{Definition}\label{def:NDeg}
If the exterior derivative $d\lambda$
does not vanish at a singular point $o$,
the point $o$ is called a {\it non-degenerate} singular point 
of the frontal $f$.
\end{Definition}

If $o$ is a non-degenerate singular point,
then the singular set of $f$ 
on a neighborhood $U$ of $o$
can be
parametrized by a regular curve on $U$.
Moreover, we may assume that there exists  a local coordinate system $(\xi,\eta)$ on 
$U$ so that $\xi$-axis is the singular set.
Such a coordinate system is called 
{\it admissible}.  Then, we can prove the following:

\begin{Lemma}\label{lem:450}
 If two admissible coordinate systems at $o$ of the frontal $f$
 have common orientation and the orientation of the singular curve,
 one of them can be smoothly deform to the other.
\end{Lemma}

\begin{proof}
We may assume that $f(u,v)$ is written in the admissible 
parametrization.
Then a given local coordinate system
$(\xi,\eta)$ gives also an admissible parametrization for $f$
if and only if
$u(\xi,\eta)$ and $v(\xi,\eta)$ satisfy
\begin{align}\label{eq:A184}
&u(0,0)=0,\quad v(\xi,0)=0,
\\ \nonumber
&
u_\xi(\xi,0)>0, \quad
u_\xi(\xi,0) v_\eta(\xi,0)-u_\eta(\xi,0) v_\xi(\xi,0)
 =u_\xi(\xi,0)v_\eta(\xi,0)>0.
\end{align}
In particular, there exist two smooth
functions $a(\xi,\eta)$ and
$b(\xi,\eta)$ so that
$$
u(\xi,\eta)=\alpha \xi+a(\xi,\eta), \qquad
v(\xi,\eta)=\beta \eta+b(\xi,\eta),
$$
where $\alpha, \beta$  are positive real numbers
and 
\begin{align}\label{eq:484}
  &a(0,0)=0,\quad b(\xi,0)=0,\\
  &\alpha+a_\xi(\xi,0)>0,\quad (\alpha+a_{\xi}(\xi,0))(\beta+b_{\eta}(\xi,0))>0
\nonumber
\end{align}
For $t\in [0,1]$, we set
$$
u^t(\xi,\eta):=\alpha \xi+t a(\xi,\eta), \quad
v^t(\xi,\eta):=\beta \eta+t b(\xi,\eta).
$$
By \eqref{eq:484},
 $(u^t,v^t)$ ($t\in I$) also induces 
an admissible parametrization of $f$ for each $t\in [0,1]$, 
and gives the desired smooth deformation from $(u,v)$
to $(\xi,\eta)$.	
\end{proof}

\begin{Definition}\label{def:GS}
Let $f:U\to M^3$ be a frontal.
We let
$(u,v)$ be an admissible coordinate system
such that each $(u,0)$
is a non-degenerate singular point of $f$
(cf. Definition \ref{def:NDeg}).
Then there exists a vector field 
$$
\eta(u)=
\alpha(u)\frac{\partial}{\partial u}+
\epsilon(u)\frac{\partial}{\partial v}
\qquad (\alpha(0)^2+\epsilon(0)^2\ne 0),
$$
along the $u$-axis such that $df_{(u,0)}(\eta(u))=\mb{0}$,
which is called a {\it null vector field} along $f$. 
If $\epsilon(u)\ne 0$, then $p:=(u,0)$ is called
a {\it singular point of the first kind}, and
otherwise it is called
a {\it singular point of the second kind}
(if $o$ is of the second kind, then 
we can choose
$\eta$ so that $\alpha(u)$ is identically equal to $1$
by a suitable scalar multiplication).
A singular point $p\in U$ of the second kind
is called a {\it generalized swallowtail} if
there exists an open neighborhood $V(\subset U)$ of $p$
such that $V\setminus \{p\}$ consists of
regular points of $f$ or  singular points of the first kind
with respect to the map $f$.
\end{Definition}

The property that a given smooth map $f$ is
a wave front or not (resp. a frontal or not)
does not depend on the choice of the Riemannian metric $g$
(cf. \cite[Chapter 10]{SUY2}).

\begin{Fact}[{\cite{SUY2})}]\label{fact544}
Let $f:U\to M^3$ be a wave front at $o$. Then
$p\in U$
is a cuspidal edge $($resp.
a swallowtail$)$ singular point
if and only if it is a singular point of the first $($resp. the second$)$
kind.
\end{Fact}

So we obtain the following assertion:

\begin{Prop}\label{prop544}
Let $f:U\to M^3$ be a generalized swallowtail at $o$. 
Then $o$ is a swallowtail singular point if and only if
$f$ is a wave front at $o$.
\end{Prop}

We denote by $\mc F^{GS}(M^3)$ the set of germs of
generalized swallowtails with admissible parametrizations (cf. Definition~\ref{def:187}).
By definition, we have
$
\mc F^S(M^3)\subset \mc F^{GS}(M^3).
$
We now fix a generalized swallowtail $f:U\to M^3$ belonging to $\mc F^{GS}(M^3)$.
Since $f$ is a frontal, there exists
a smooth unit normal vector field $\nu(u,v)$
defined on $U$.
We set
\begin{equation}\label{eq:lambda}
\lambda:=\op{det}_g(f_u,f_v,\nu).
\end{equation}
Since the $u$-axis is the singular set,
we have $\lambda_u(u,0)=0$.
Since $o$ is a singular point of the second kind,
$f_u(o)$ vanishes and $d\lambda$ does not vanish  at $o$.
So it holds that
$$
0\ne \lambda_v(o)=\op{det}_g(\nabla_vf_u(o),f_v(o),\nu(o)).
$$
In particular, we have
\begin{equation}\label{eq:448}
\nabla_vf_u(o)=\nabla_uf_v(o)\ne \mb 0.
\end{equation}
By Corollary \ref{cor:A1}
in the appendix,
we assume that our unit normal vector field $\nu$ satisfies
\begin{equation}\label{eq:nu_C}
\op{det}_g(\nabla_vf_u(o),f_v(o),\nu(o))>0,
\end{equation}
that is, $\{\nabla_vf_u(o),f_v(o),\nu(o)\}$ is a positive frame at $o$
(cf. (2) of Corollary \ref{cor:A1}).
Since (the definition of $\inner{}{}$ is in \eqref{eq:422})
\begin{equation}\label{eq:612}
\inner{\nabla_uf_v(o)}{\nu(o)}=-\inner{f_u(o)}{\nabla_v\nu(o)}=0,
\end{equation}
$\nu(o)$ is perpendicular to $\nabla_vf_u(o)$ and $f_v(o)$. 
In particular, $\nu(o)$ can be obtained by
a positive scalar multiplication of the vector
$
\nabla_vf_u(o)\times_g f_v(o) 
$.
Using this $\nu$, we set
\begin{equation}\label{eq:sigma0}
\sigma^S_{0}:=-\op{sgn}\Big(
\inner{\nabla_u f_{v}(o)}{\nabla_u \nu(o)}\Big) 
\end{equation}
which corresponds to the numerator of
 \cite[(4.8)]{MSUY} (see also Remark \ref{rmk:617})
and has the same sign as the normalized 
cuspidal curvature $\mu_C$
by adjusting the $\pm$-ambiguity 
of the unit normal vector field $\nu$ of $f$.

\begin{Rmk}\label{rmk:617}
In \cite[Definition 4.1]{MSUY},
the local coordinate system $(u,v)$ with respect
to a given generalized swallowtail $f$ at $(u,v)=o$
$(o:=(0,0))$
satisfies the following three properties;
\begin{enumerate}
\item $f_u(0,0)=\mb 0$,
\item the $u$-axis corresponds to the singular set of $f$,
\item $f_v(p)$ is a unit vector.
\end{enumerate}
However, the formulas
\begin{equation}\label{eq:629}
\mu_C(o)=-
\frac{|f_v(o)|^3\inner{(\nabla_v f_u)(o)}{\nu_u(o)}}
{|f_{uv}(o)\times_g f_v(o)|} 
\end{equation}
given in \cite[(4.8)]{MSUY} do not
use the assumption (3).
Since our admissible coordinate systems (cf. Definition 0.1)
satisfy (1) and (2),
we can apply \eqref{eq:629} in our setting.
\end{Rmk}

\begin{Prop}
Let $f:U\to M^3$ be
a generalized swallowtail belonging to $\mc F^{GS}(M^3)$.
Then $\sigma^S_{0}\ne 0$ if and only if $f(u,v)$
has  a swallowtail singular point at $o$.
In particular,
$\sigma^S_{0}\in \{1,-1\}$ holds if $o$ is a swallowtail singular point.
\end{Prop}

\begin{proof}
By Fact \ref{fact544}, it is sufficient to show that
$\sigma^S_{0}\ne \bf 0$ if and only if $f(u,v)$
is a wave front at $o$:
We have seen that $\nu$ is perpendicular to $\nabla_{u}f_v$ at $o$,
which implies
$$
0=\inner{\nabla_uf_v(o)}{\nu(o)}=-\inner{f_v(o)}{\nabla_u\nu(o)}.
$$
So $\nabla_u\nu(o)$ is perpendicular to $f_v(o)$.
Since $\nabla_u\nu(o)$ is perpendicular to $\nu(o)$,
$\sigma^S_{0}\ne \mb 0$ holds if and only if $\nabla_u\nu(o)\ne 
\mb 0$.
Since $f_u(o)$ vanishes, it can be easily observe that
$\nu\colon{}U\to T_1M^3$ is an immersion at $o$ if and only 
if  $\nabla_u\nu(o)\ne \mb 0$,
proving the assertion.
\end{proof}

As a consequence, we have the following:

\begin{Corollary}
The sign $\sigma^S_{0}$ does not depend on the choice of an admissible coordinate system
nor the choice of a Riemannian metric of $M^3$.
\end{Corollary}

\begin{Definition}\label{def:C298}
A germ of generalized swallowtail $f$ 
is said to be {\it positively oriented} (resp. {\it negatively oriented})
if $\sigma_0^S$ is positive (resp. negative) at $o$.
\end{Definition}

Since $o$ is a swallowtail singular point,
there exits a regular curve
$\gamma:(-\epsilon,\epsilon)\to U$ ($\epsilon>0$)
such that
$\gamma(0)=o$
and
each $\gamma(u)$ ($u\ne 0$) is
a generalized cuspidal edge singular point of $f$ as defined in 
Appendix ~\ref{app:0}, where
we defined
the invariant $\sigma^C_{0}(u)$.
The following assertion holds:

\begin{Proposition}
At a generalized swallowtail singular point,
$\sigma^S_0$ vanishes if $\sigma^C_{0}(u)$ $(u\ne 0)$
vanishes identically.
Conversely, if $\sigma^S_0 \ne 0$, then 
$\dy\lim_{u\to 0}\sigma^C_{0}(u)=\sigma^S_0$ holds.
\end{Proposition}

\begin{proof}
Let $f:(U;u,v)\to M^3$ be a frontal 
whose singular points are all non-degenerate
(see Definition \ref{def:NDeg}).
Then a pair $(V,\zeta)$ of non-vanishing
vector fields on $U$ 
is called an {\it admissible pair} if $V$ is
tangent to the singular curve and
$\zeta$ points in the kernel of the differential $df$
at each singular point of $f$.
We can take a vector field $\zeta$
on $U$ such that $(\partial/\partial u,\zeta)$
is an admissible pair.
As pointed out in \cite[Proposition~1.6]{SUY2022}, 
\begin{equation}\label{eq:sigmaS0}
\sigma^C_{0}(u)=\op{sgn}(\op{det}_g(f_u(u,0),
  \nabla_\zeta f_{\zeta}(u,0),\nabla_\zeta \nabla_\zeta f_{\zeta}(u,0)))
\end{equation}
holds. The statement of 
\cite[Proposition~1.6]{SUY2022} is described
for the cuspidal edge singularity, 
but can be extended to the case of generalized cuspidal edges,
since we have not use the fact that $f$ is a wave front there.
Using the formula in \cite[Page 272]{SUY2022}, 
we obtain
$$
\dy\lim_{u\to 0}\sigma^C_{0}(u)=\lim_{u\to 0}\op{sgn}(\kappa_c(u))
=\op{sgn}(\mu_C)=
\sigma^S_{0},
$$
where $\mu_C$ is the 
normalized cuspidal curvature.
This proves the assertion.
\end{proof}

In \cite[(2.5)]{MSUY}, the limiting normal
curvature $\kappa_\nu$ at the swallowtail 
singular point $o$ is defined by (cf. \cite[(2.2)]{MSUY})
\begin{equation}\label{eq:kappaN}
\kappa_\nu(o):=\frac{\inner{\nabla_v f_{v}(o)}{\nu(o)}}{\inner{f_v(o)}{f_v(o)}}.
\end{equation}
Regarding this,
we consider the sign
\begin{equation}\label{eq:sigmaG}
\sigma^S_{g}:=\op{sgn}\,\Big(
\inner{\nabla_v f_{v}(o)}{\nu(o)}\Big)
\end{equation}
and gives the following definition, which obviously
does not depend on the choice of admissible parametrizations.

\begin{Definition}\label{def:C298b}
A germ of  generalized swallowtail 
$f$ is called {\it generic}
if $\sigma_g^S$ does not vanish.
We denote by $\mathcal F_*^S(M^3)$ (resp. $\mathcal F_*^{GS}(M^3))$ 
the set of germs of generic swallowtails (resp. generic generalized swallowtails)
at $o\in \R^2$ into $M^3$ parametrized by admissible 
parametrizations.
\end{Definition}

We let 
$\gamma:(-\epsilon,\epsilon)\to U$ ($\epsilon>0$) be
a regular curve parametrizing the singular set of $f$
such that $\gamma(0)=o$.
Since $\gamma(u)$ ($u\ne 0$) is 
a generalized cuspidal edge singular point of $f$, 
the corresponding invariant $\sigma^C_{g}(u)$
is defined in Appendix \ref{app:0}.

\begin{Proposition}
\label{prop:517}
If $\sigma^S_g$ does not vanish, then
it coincides with the sign $\sigma^C_{g}(u)$
for sufficiently small $|u|(\ne 0)$.
\end{Proposition}

\begin{proof}
We use the notations of 
Appendix \ref{App:A}.
By (1) of
Corollary \ref{cor:A1},
the sign $\sigma^C_{g}(u)$ coincides with
the sign of
$$
\op{det}_g(f_u(u,0),f_{\tilde \eta\tilde \eta}(u,0),\nabla_uf_{u}(u,0))
=\inner{\tilde \nu(u,0)}{\nabla_uf_{u}(u,0)},
$$
where $\tilde \nu(u,0):=f_u(u,0)\times_g f_{\tilde \eta\tilde \eta}(u,0)$.
By Proposition \ref{prop:A2}, $\sigma_g^S$
coincides with the sign of
$
\inner{\nabla_vf_{v}(u,0)}{\tilde \nu(u,0)}
$, proving the assertion.
\end{proof}

Since
$\sigma^C_{g}(u)$ vanishes if and only if the limiting normal
curvature at $(u,0)$ vanishes,
Proposition \ref{prop:517} implies
the following:

\begin{Corollary}\label{def:S490}
A germ of swallowtail $f$ is asymptotic
$($cf. Definition \ref{def:asymptotic}$)$
if and only if $\sigma^C_{g}(u)$ vanishes for each 
$u\in (-\epsilon,\epsilon)\setminus\{0\}$.
\end{Corollary}

\subsection*{The case that $M^3$ is a space form}
 
The invariant $\sigma^S_g$ for generalized swallowtails 
depends on the Riemannian metric of $M^3$, in general.
However, by the stereographic projection,
we can give a parametrization of the
space from $M^3(a)$ so that
the invariant $\sigma^S_{g}$ for a given swallowtail
does not depend on $a\in \R$:
Like as in the case of cuspidal edges (cf. \cite{SUY2022}),
we consider the Riemannian metric
\begin{equation}\label{eq:gA822}
g_a=\rho^2 g_E \qquad \left(\rho:=\frac{2}{1+a (x^2+y^2+z^2)}\right),
\end{equation}
where $g_E:=dx^2+dy^2+dz^2$ is the canonical flat metric on $\R^3$,
and set
$$
\R^3(a):=
\begin{cases}
\R^3 & (\text{if $a\ge 0$}), \\
\{(x,y,z)\in \R^3\,;\, x^2+y^2+z^2<1/\sqrt{|a|}\} & (\text{$a<0$}).
\end{cases}
$$
Then $(\R^3(a),g_a)$ is of constant sectional curvature $a$ (cf. \cite{UY80}),
which gives a parametrization of space forms.
Moreover, this has the desired property  as follows:

\begin{Proposition}\label{lem:852}
Let $f:U\to \R^3(a)$ be a generalized cuspidal edge in $\R^3(a)$.
Then the sign $\sigma^C_{g_a}(u)$ does not depend on $a \in \R$.
\end{Proposition}

\begin{proof}
If $f$ is a cuspidal edge, this statement has been proved in
{\cite[(2.1)]{SUY2022}}. However, the argument there 
depends only on the fact that $f$ is a generalized cuspidal edge
but was not used the fact that it is a cuspidal edge.
So we obtain the  conclusion.
\end{proof}

Similarly, we obtain the following:

\begin{Proposition}\label{prop:610}
Let $f:U\to \R^3(a)$ be a generalized swallowtail at $o$
belonging to $\mc F^{GS}(\R^3(a))$.
Then the invariant $\sigma^S_{g_a}$
does not depend on $a$.
\end{Proposition}

\begin{proof}
Since $f(o)=\mb 0$,
$f_{uv}(o):=(\nabla_uf_v)(o)$ and
$\nu_{u}(o)=\nabla_u\nu(o)$ hold, 
since 
the all of the Christoffel symbols of 
the metric $g_a$ (cf. \eqref{eq:gA822})
at the origin of $\R^3(a)$ vanishes simultaneously.
So
$$
\inner{\nabla_v f_{v}(o)}{\nu(o)}=
\inner{f_{vv}(o)}{\nu(o)}.
$$
Since the inner product $g_a:=\inner{}{}$ at the origin $o$ does not 
depend on $a$, we obtain the conclusion.
\end{proof}

By Proposition \ref{prop:610},
to prove the propositions and 
theorems in the introduction,
we may assume that $a=0$.
Regarding this fact, we set
$$
\sigma^S:=\sigma^S_{g_a} \qquad (a\in \R),
$$
since the right-hand side does not depend on $a$.

\begin{Proposition}\label{prop:903}
Let $f:U\to \R^3(a)$ be a generalized swallowtail at $o$
belonging to $\mc F^{GS}(\R^3(a))$.
Then the Gaussian curvature $K_{a}$ of $f$ in $\R^3(a)$
is related to the Gaussian curvature $K$ of $f$ by thinking that $f$
lies in $\R^3(0)$ by
\begin{equation}\label{eq:908}
K_a=a+(\rho\circ f)K_{0}.
\end{equation}
In particular, the extrinsic Gaussian curvature $K_{ext}$ of $f$ in $\R^3(a)$
coincides with $(\rho\circ f)K_{0}$.
As a consequence, the positivity and the negativity of
asymptotic swallowtails in $\mathcal F_0^S(M^3(a))$ 
does not depend on the choice of $a\in \R$.
\end{Proposition}

\begin{proof}
Since $f$ has co-rank one singular points along the $u$-axis,
we may assume that there exists a vector field 
$\mb e_1$ defined on $U$ such that $df(\mb e_1)$ 
is a unit vector at each point of $U$.
If we set 
$$
\hat{\mb e}_1:=df(\mb e_1),\qquad
\hat{\mb e}_2:=\nu\times_{g_a} \hat{\mb e}_1,
$$
then $\{\hat{\mb e}_1,\hat{\mb e}_2,\nu\}$ 
forms an orthonormal frame field
along $f$. Since $f$ is an immersion on $U_*:=U\setminus \{v=0\}$,
there exists a unique vector field $\mb e_2$ on $U_*$
satisfying $df(\mb e_2)=\hat{\mb e}_2$.
In particular,
$
\mb e^E_1:=\rho \mb e_1,\,\,
\mb e^E_2:=\rho \mb e_2
$
are unit vector fields  satisfying
$g_0(\mb e^E_1,\mb e^E_2)=0$ on $U_*$.
Then, the identity
$$
K_a=a+ h(\mb e_1,\mb e_1)
h(\mb e_2,\mb e_2)
-h(\mb e_1,\mb e_2)^2
$$
holds on $U_*$ (cf. \cite[Proposition 4.5, Chapter 7]{KN}),
where $h$ is the second fundamental form of $f$ in $\R^3(a)$.
Then by the same computation as in the proof of \cite[Section~2]{SUY},
we have  \eqref{eq:908} on $U_*$.
\end{proof}

\section{A representation formula for swallowtails in $\R^3$}

In this section, we introduce a formula which produces
all germs of generic swallowtails in the space-form 
$
\R^3(a)
$
with the Riemannian metric $g_a$.
We first define space-cusps as follows.

\begin{Definition}\label{def:gen348}
Let $I$ be an open interval containing $0$ in $\R$
and $\Gamma:I\to \R^3(a)$ a $C^\infty$-map.
Then $\gamma(u)$ is called an {\it space-cusp} at $u=0$ if
$\gamma'(0)=\mb 0$ and two vectors $\gamma''(0),\,\gamma'''(0)$ are
linearly independent in $\R^3$,
where ${}'$ means $d/du$.
Moreover, if $\gamma''(0),\,\gamma'''(0),\gamma^{(4)}(0)$
are linearly independent, $\gamma$ is said to be {\it generic} (at $0$).
\end{Definition}

If $\gamma'(0)=\mb 0$ and $\gamma''(0)\ne \mb 0$, we can write
\begin{equation}\label{eq:gug243}
\gamma'(u)=u\xi(u),
\end{equation}
where $\xi(u)$ is an $\R^3$-valued function such that $\xi(0)\ne \mb 0$.
We have the following equivalence relations;
\begin{align}\label{A}
\text{$\gamma(u)$ is a space-cusp at $u=0$}\,\,
&\Longleftrightarrow \,\, \xi(0)\times \xi'(0)\ne \mb 0, \\
\label{B}
\text{$\gamma(u)$ is a generic space-cusp at $u=0$}\,\,
&\Longleftrightarrow \,\, \op{det}(\xi(0),\xi'(0),\xi''(0))\ne 0, 
\end{align}
where $\times$ is the canonical vector product on $\R^3$.
The conditions
\eqref{A} and \eqref{B}
are both independent of the choice of a parametrization of $\gamma$.

\begin{Remark}\label{rmk:unit}
By 
Proposition~\ref{prop:B1855} in Appendix \ref{App:B},
we may assume that $|\xi(u)|=1$ for each $u$,
where $|\mb a|:=\sqrt{g_0(\mb a,\mb a)}$
for $\mb a\in \R^3$. 
\end{Remark}

\begin{Remark}\label{rmk:998}
The definition of space-cusp in coordinate-invariant form 
is as follows: A $C^\infty$-map $\gamma:I\to M^3(a)$ is called 
a {\it space-cusp} if 
\begin{enumerate}
\item $\gamma'(0)$ ($u\in I$) vanishes if and only if $u=0$,
\item $\nabla_u\gamma'(0)$ and $\nabla_u\nabla_u\gamma'(0)$
are linearly independent in $T_{\gamma(0)}M^3(a)$.
\end{enumerate}
Moreover, if
$\nabla_u\gamma'(0)$, $\nabla_u\nabla_u\gamma'(0)$
and $\nabla_u\nabla_u\nabla_u\gamma'(0)$ are linearly independent
in $T_{\gamma(0)}M^3(a)$,
$\gamma$ is called a {\it generic space-cusp}.
Since $\gamma'(0)=\mb{0}$,
as explained in \cite[Page 72]{SUY2022}, we have
$$
\nabla_u\gamma'(0)=\gamma''(0), \quad
\nabla_u\nabla_u\gamma'(0)=\gamma'''(0).
$$
Moreover, since 
the all of Christoffel symbols of 
the metric $g_a$ at the origin of $\R^3(a)$ vanishes simultaneously.
We have
$$
\nabla_u\nabla_u\nabla_u\gamma'(0)=\gamma^{(4)}(0).
$$
So this definition is compatible with Definition \ref{def:gen348}.
\end{Remark}

\begin{Definition}\label{def:gen431}
A generic space-cusp $\gamma(u)$ at $u=0$
is said to be {\it right-handed} (resp. {\it left-handed}) if
$\det(\xi,\xi',\xi'')$ is positive (resp. negative) at $u=0$.
\end{Definition}

\begin{Remark}\label{rmk:pm432}
If $\gamma(u)$ is a positive (resp.~negative) space-cusp,
then $\gamma(-u)$ and $-\gamma(u)$ are both negative (resp.~positive) space-cusps.
\end{Remark}

We show the following:

\begin{Prop}
Let $f:U\to \R^3(a)$ 
be a generalized swallowtail in $\mc F^{GS}(\R^3(a))$.
Then $\gamma(u):=f(u,0)$  is a space-cusp $\gamma(u)$ at $u=0$.
\end{Prop}

\begin{proof}
We let $\nu$ be a unit normal vector field of $f$ on $U$, and
denote by the plane $\Pi_o$ in $\R^3(a)(\subset\R^3)$ passing through 
$f(o)$ which is
perpendicular to $\nu(o)$. 

The plane $\Pi_o$ is called the {\it limiting tangent plane}
 of $f$ at $o$.
We denote by $\pi:\R^3\to \Pi_o$ the canonical orthogonal
projection. If we set $\hat f:=\pi\circ f$,
then we have
$$
\hat f(u,v)=f(u,v)-\inner{f(u,v)}{\nu(o)}\nu(o).
$$
We then set 
$$
\lambda:=\det(f_u,f_v,\nu),\qquad  \hat\lambda:=\det(\hat f_u,\hat f_v,\nu(o)).
$$
Since $f_u(o)=\mb 0$, we have
$
\hat \lambda_v(o)=\det(\hat f_{uv}(o),
\hat f_v(o),\nu(o))
$
and
$$
\hat f_{uv}(o)
=f_{uv}(o)-\inner{f_{uv}(o)}{\nu(o)}\nu(o).
$$
By \eqref{eq:612}, $\inner{f_{uv}(o)}{\nu(o)}$ vanishes,
and $\hat f_{uv}(o)$ coincides with $f_{uv}(o)$.
Thus, we have
$$
\hat \lambda_v(o)=\lambda_v(o).
$$
Since $\lambda(u,0)$ vanishes identically, we have $\lambda_u(o)=0$.
Since $o$ is a 
non-degenerate singular point (cf. Definition \ref{def:NDeg}), $\lambda_u(o)=0$
implies $\lambda_v(o)\ne 0$. So $\hat \lambda_v(o)$ does not vanish.
So, $\hat f:U\to \Pi_a(=\R^2)$ has a Whitney cusp at $o$,
which implies that $\hat \gamma(u):=\pi\circ \gamma(u)$
is a cusp.
Since
$$
\hat \gamma=\gamma-\inner{\gamma}{\nu(o)}\nu(o),
$$ 
we have that
$$
\hat \gamma''(0)=\gamma''(0)-\inner{\gamma''(0)}{\nu(o)}\nu(o),
\qquad
\hat \gamma'''(0)=\gamma'''(0)-\inner{\gamma'''(0)}{\nu(o)}\nu(o)
$$ 
and so
$$
0\ne \det\Big(\hat \gamma''(0),\hat \gamma'''(0),\nu(o)\Big)=
\det\Big(\gamma''(0),\gamma'''(0),\nu(o)\Big),
$$
which implies that $\gamma(u)$ gives a space-cusp at $u=0$.
\end{proof}

We now fix
a generalized swallowtail $f:U\to \R^3(a)$ belonging to $\mc F^{GS}(\R^3(a))$.
By the division lemma (cf. \cite[Appendix~A]{SUY2}), 
we can write
$f(u,v)-f(u,0)=v\mb a(u,v)$,
where $\mb a(u,v)$ is an $\R^3$-valued $C^\infty$-function.
Since $f(u,0)=\gamma(u)$,
we have
$$
f(u,v)=\gamma(u)+v \mb a(u,v)=
\gamma(u)+v \Big(\mb a(u,v)-\mb a(u,0)\Big)+v \mb a(u,0).
$$
By applying the division lemma again, we can write
$$
\mb a(u,v)-\mb a(u,0)=v\mb b(u,v),
$$
where $\mb b(u,v)$ is an $\R^3$-valued $C^\infty$-function.
Since $\gamma(u)$ is a space-cusp, 
we can write $\gamma'(u)$ as in \eqref{eq:gug243}.
We set (cf. \eqref{eq:lambda})
\begin{equation}\label{eq:lambda337}
\tilde \lambda:=\det(f_u,f_v,\tilde\nu),
\end{equation}
where $\tilde \nu(u,v)(\ne \mb 0)$ is a normal vector field of $f$
defined on a neighborhood of the origin $o$.
Since the $u$-axis  is a singular point of $f$,
we have
$$
0=\tilde \lambda(u,v)|_{v=0}=
\det\Big(u\xi(u),\mb a(u,0),\tilde \nu(u,0)\Big)
=u\det\Big(\xi(u),\mb a(u,0),\tilde\nu(u,0)\Big).
$$
In particular,
$
\det(\xi(u),\mb a(u,0),\tilde\nu(u,0))
$
vanishes for $u\ne 0$. By the continuity, we have that
$$
\det\Big(\xi(0),\mb a(0,0),\tilde\nu(0,0)\Big)=0.
$$
So, $\xi(u)$ and $\mb a(u,0)$ are linearly dependent 
for each $u\in I$.
So we  can write
$
\mb a(u,0)=\alpha(u)\xi(u)
$,
where $\alpha(u)$ is a germ of $C^\infty$-function at $u=0$.
Since $o$ is a non-degenerate singular point of the frontal $f$, 
it is a corank one singular point (cf. \cite{SUY2}).
So the fact $f_u(o)=\mb 0$ implies
\begin{equation}\label{eq:fv354}
\mb 0\ne f_v(o)=\mb a(o)=\alpha(0)\xi(0).
\end{equation}
So $\alpha(0)\ne 0$, and 
we can replace $v$ by $v\alpha(u)$. Thus, $f$ has an expression of the following form;
\begin{equation}\label{eq:f265}
f(u,v)=\gamma(u)+v\xi(u)+v^2 \mb b(u,v).
\end{equation}
By setting
$\xi'=d\xi/du$ and $\xi''=d^2\xi/du^2$,
we have
\begin{equation}\label{eq:fuv482}
f_u=u\xi +v(\xi'+v\mb b_u),\qquad
f_v=\xi+v(2 \mb b+v\mb b_v)
\end{equation}
and
$$
f_u(u,v)\times f_v(u,v)=v\Big(
-\xi(u)\times \xi'(u)+2u \xi(u)\times \mb b(u,0)+O(v)\Big),
$$
where $O(v)$ 
is a term written in $v \phi(u,v)$
for a $C^\infty$-function germ $\phi(u,v)$ at $o\in \R^2$. 
Moreover, we have
\begin{equation}\label{eq:fvv}
f_{uv}(o)=\xi'(o).
\end{equation}
By \eqref{A},
\begin{equation}\label{eq:nu389}
\tilde \nu(u,v)=-\xi(u)\times \xi'(u)+2u \xi(u)\times \mb b(u,0)+O(v)
\end{equation}
is the normal vector field $\tilde \nu$ of $f$. 
By \eqref{eq:fuv482}, we have
$$
\xi'(u)=\lim_{v\to 0} \frac{f_u(u,v)-uf_v(u,v)}{v},
$$
and so $\tilde \nu(u,0)$ is perpendicular to both $\xi(u)$ and $\xi'(u)$.
We set
$
\nu(u,v):={\tilde \nu}/{|\tilde \nu|},
$
which gives the unit normal vector field of $f$.
Moreover, one can easily check that $\nu$
satisfies \eqref{eq:nu_C},
and
\begin{equation}\label{eq:nuu738}
\nu_u(u,0)
=
\frac{-\xi(u)\times \xi''(u)+2\xi(u)\times \mb b(u,0)}{|\tilde \nu(u,v)|}
+\left.\left(\frac{1}{|\tilde \nu(u,v)|}\right)_{\!\!u}\right|_{v=0} \tilde \nu(u,0)
\end{equation}
holds,
which implies that $\nu_u(u,0)$ is perpendicular to $\xi(u)$.
So $\nu_u(u,0)$ and $\nu(u,0)\times \xi(u)$ are linearly dependent.
Thus $\nu_u(u,0)\ne \mb 0$ if and only if
\begin{equation}\label{eq:512}
\delta(u):=\inner{\tilde \nu_u(u,0)}{\tilde \nu(u,0)\times \xi(u)}
\end{equation}
does not vanish. By the well-known formula on vector product
(cf. \cite[Proposition~A.3.7]{UY}), we have
$$
\tilde \nu(u,0)\times \xi(u)=
-\Big(\xi(u)\times \xi'(u)\Big)\times \xi(u)=
\inner{\xi(u)}{\xi'(u)}\xi(u)-|\xi(u)|^2\xi'(u).
$$
Since $\tilde \nu_u(u,0)$ is orthogonal to $\xi(u)$ and $\xi'(u)$,
we have
\begin{align*}
\delta(u)&=
|\xi(u)|^2 \inner{\tilde \nu_u(u,0)}{\xi'(u)} \\
&=\frac{|\xi(u)|^2}{{|\tilde \nu(u,0)|}}\inner{
-\xi(u)\times \xi''(u)+2\xi(u)\times \mb b(u,0)}{ 
\xi'(u)} \\
&=\frac{|\xi(u)|^2}{{|\tilde \nu(u,0)|}}
\det\Big(\xi(u),\xi'(u), -\xi''(u)+2\mb b(u,0)\Big).
\end{align*}
Here, $f$ is a wave front at $o$ 
if and only if the matrix
$$
\pmt{f_u(o)& f_v(o) \\
     \nu_u(o) & \nu_v(o)}=
\pmt{\mb 0 & f_v(o) \\
     \nu_u(o) & \nu_v(o)}
$$
is of rank $2$. 
Since $\nu_u(o)\ne \mb 0$
holds if and only if $\delta(0)\ne \mb 0$,
the map-germ $f$ is a wave front if
and only if
\begin{equation}\label{eq:delta394}
\varDelta_0^S:=
-\det\Big(\xi(0),\xi'(0),-\xi''(0)+2\mb b(o)\Big)
\end{equation}
does not vanish. 
The following assertion holds:

\begin{Proposition}
The sign of $\varDelta_0^S$ coincides with $\sigma^S_0$.
\end{Proposition}

\begin{proof}
In fact, this is obtained by
substituting
\eqref{eq:fvv} and
\eqref{eq:nuu738}
into \eqref{eq:sigma0}.
In fact, $f(o)=\mb 0$,
$f_{uv}(o):=(\nabla_uf_v)(o)$ and
$\nu_{u}(o)=\nabla_u\nu(o)$ hold, since 
the all of Christoffel symbols of 
the metric $g_a$ (cf. \eqref{eq:gA822})
at the origin of $\R^3(a)$ vanishes simultaneously.
\end{proof}

\begin{Remark}
If $f\in F^S(\R^3)$, then
the sign $\sigma_0^S$ coincides with $\op{sgn}_{\triangle}(\nu)$ given in \cite[Page 519]{SUY}.
By the coordinate change
$(u,v)\mapsto (-u,-v)$, the sign  $\sigma_0^S$ reverses.
So we may assume that $\sigma_0^S$ is positive.
\end{Remark}

By \eqref{eq:lambda337} and
\eqref{eq:nu389},
we have that
$$
\lambda_u(u,v)=v |\tilde \nu(u,v)|^2+O(v^2),\qquad
\lambda_v(o)=|\xi(0)\times \xi'(0)|^2\ne \mb 0,
$$ 
where $O(v^i)$ ($i=1,2,3,\ldots$)
is a term written in $v^i \phi(u,v)$
for a $C^\infty$-function germ $\phi(u,v)$ at $o\in \R^2$. 
So, $o$ is a non-degenerate singular point of $f$.
Moreover, the null direction of $f$
is given by $\zeta:=\partial_u-u \partial_v$
(i.e. $f_u-u f_v=\mb 0$ along the $u$-axis), 
we have 
$$
\lambda_{\zeta}(u,v)= 2v \inner{\tilde\nu(u,v)}{\tilde\nu_u(u,v)}-u |\tilde\nu(u,v)|^2-2uv 
\inner{\tilde\nu(u,v)}{\tilde\nu_v(u,v)}
$$
and
$$
\lambda_{\zeta\zeta}(o)=
\lambda_{\zeta u}(o)=-|\tilde\nu(o)|^2 \ne 0.
$$
By \cite[Theorem~4.2.3]{SUY2}, the origin
$o$ is a swallowtail singular point of $f$
if and only if $\sigma_0^S$ given in
\eqref{eq:delta394}
does not vanish. So, we obtain the following:

\begin{Theorem}\label{prop:SW-rep}
Let $\gamma(u)$ be a germ of space-cusp at $u=0$
in $\R^3(a)$ satisfying $\gamma'(u)=u\xi(u)$
and let $\mb b(u,v)$ be a germ of
$\R^3$-valued $C^\infty$-function at $(u,v)=(0,0)$.
Then 
\begin{equation}\label{eq:f265b}
f(u,v):=\gamma(u)+v\xi(u)+v^2 \mb b(u,v)
\end{equation}
gives a germ of generalized swallowtail at $o\in \R^2$
$($we call such an $f$ the {\em generalized swallowtail induced by $\xi$ and $\mb b)$}.
Conversely, any germ of generalized swallowtails at $o$
can be obtained in this manner.
Moreover,
if  $\sigma_0^S$ $($resp. $\sigma^S)$
does not vanish, 
$f$ is a germ of
swallowtail $($resp. generic generalized swallowtail$)$
at $o\in \R^2$.
\end{Theorem}

\begin{Remark}
A similar but different representation formula
for swallowtails is given in Fukui \cite{F} and the first author \cite{S}.
Comparing \eqref{eq:f265} with the formula in \cite{F},
ours has ambiguities of the choice of 
the parameters $u,v$. On the other hand,
each of the parameters $u$ and $v$ of $f$
in the formula in Fukui \cite{F}
has a special geometric meaning. 
\end{Remark}

\begin{proof}[Proof of Theorem C]
By \eqref{eq:delta394}, we have
\begin{equation}\label{eq:1220}
-\sigma^S_0=
\op{sgn}\Big(
\det\Big(\xi(0),\xi'(0),-\xi''(0)+2\mb b(o)\Big)\Big).
\end{equation}
Since $f$ is non-generic, 
$$
\sigma^S_g=
\op{sgn}\left(
\det\Big(\xi(0),\xi'(0),\mb b(o)\Big)\right)
$$
vanishes. This with \eqref{eq:1220}, we have
\begin{equation}\label{eq:1232}
\sigma^S_0=
\op{sgn}\left(
\det\Big(\xi(0),\xi'(0),\xi''(0)\Big)\right),
\end{equation}
which implies the conclusion.
\end{proof}

Proof of Corollary C${}'$ will be given in the next subsection.

\subsection*{Generalized swallowtails with $\sigma^S\ne 0$}

As we defined in the introduction, 
a swallowtail singular point $o$ is said
to be {\it generic} (cf. Definition \ref{def:227})
(resp. {\it asymptotic}, see Definition \ref{def:asymptotic}) 
if $\sigma^S$ does not vanish
(resp. $\sigma^C_g$ for cuspidal edges near $o$ vanishes identically).
For swallowtail singular points and cuspidal edge 
singular points, we can define
the limiting normal curvature $\kappa_\nu(u)$ 
along the $u$-axis (as the singular set, cf. \cite{MSUY}).
Then $\kappa_\nu(u)$ does not vanish at $u=0$
(resp. vanishes identically along the $u$-axis).

\begin{Proposition}
Let $f(u,v)$ be a germ of swallowtail at $o$
in $\R^3(a)$ as in Theorem~\ref{prop:SW-rep}.
Then $\sigma^S$ coincides with
the sign of the value\footnote{
The sign
$\sigma^S$ coincides with 
$\op{sgn}_0(\nu)$ given in \cite[(3.13)]{SUY}.
}
$$
\varDelta_1^S:=\det\Big(\xi(0),\xi'(0),\mb b(o)\Big).
$$
Moreover, if $\sigma^S>0$ $($resp. if $\sigma^S<0)$,
then the  sign of the Gaussian curvature on the tail-part 
of $f$ coincides 
with the product $\sigma_0^S\sigma^S$, where $\sigma_0^S$ is 
given in \eqref{eq:delta394}
$($``tail-part" is defined in \cite{SUY}, which is the side of singular set
without self-intersections$)$.
\end{Proposition}

\begin{proof}
By \cite[(2.2)]{MSUY}, the sign of $\kappa_{\nu}$ 
coincides with that of
$
f_{vv}(o)\cdot \nu(o).
$
By \eqref{eq:nu389}, it has the same sign as
\begin{equation}\label{eq:901}
\mb b(o)\cdot \Big(\xi(0)\times \xi'(0)\Big)
=\op{det}\Big(\xi(0),\xi'(0),\mb b(o)\Big),
\end{equation}
which implies the first assertion.
The second assertion follows from \cite[Theorem~3.11]{SUY}.
\end{proof}

Since the swallowtail singular point $o$ of $f$
is asymptotic if and only if $\varDelta_1^S$ vanishes
identically on the $u$-axis,
by replacing $(o)$ by $(u,0)$ in the above proof,
the following assertion is obtained:

\begin{Corollary}\label{cor:682}
A germ of swallowtail in $\R^3(a)$
given in Theorem~\ref{prop:SW-rep} is asymptotic
if and only if $\mb b(u,0)$ 
can be written as
a linear combination of $\xi(u)$ and $\xi'(u)$
at each $u$.
\end{Corollary}

\begin{Proposition}\label{prop:937}
For any generic $($resp. non-generic$)$
space-cusp $\gamma$ in $\R^3(a)$, there exists a 
swallowtail $f$ along $\gamma$
$($that is, $\gamma(u)=f(u,0))$ 
satisfying $\sigma_0^S>0$ 
 and $\sigma^S>0$ or
$\sigma^S<0$ $($resp. $\sigma^S=0)$.
\end{Proposition}

\begin{proof}
We set $\gamma'(u)=u\xi(u)$.
By replacing $\xi(u)$ by $\xi(-u)$ if necessary,
we may assume that
$  \det(\xi,\xi',\xi'')$ is positive.
If $\gamma$ is generic (resp. non-generic), then
we set $\mb b:=\xi''$ or $\mb b:=-\xi''/4$ (resp. $\mb b:=-\xi''/2$).
Then
$\det(\xi,\xi',\xi''+2\mb b)$
is positive and $\det(\xi,\xi',\mb b)$
is also positive or negative (resp. equal to zero).
By Theorem~\ref{prop:SW-rep},
the swallowtail  $f$ associated with 
$\mb b=\xi''$ or  $\mb b=-\xi''/4$ (resp. $\mb b:=-\xi''/2$) exists, 
which satisfies $f(u,0)=\gamma(u)$,
$\sigma_0^S>0$ and $\sigma^S>0$ or  
$\sigma^S<0$ (resp. $\sigma^S=0$).
\end{proof}

If $\gamma$ is non-generic, we obtain the following 
restriction of the sign $\sigma_0^S\sigma_1^S$:

\begin{Proposition}\label{prop:966}
Let $\gamma$ be a non-generic space-cusp in $\R^3(a)$. 
If a swallowtail $f$ along $\gamma$ exists,
then it satisfies $\sigma_0^S\sigma_1^S<0$.
Conversely, for any non-generic space-cusp $\gamma$,
there exists a generic swallowtail $f$ along $\gamma$
satisfying $\sigma_0^S\sigma_1^S<0$.
\end{Proposition}

\begin{proof}
We may set $\gamma'(u)=u\xi(u)$.
Since $\gamma$ is non-generic,
$\det(\xi(0),\xi'(0),\xi''(0))$ is equal to zero.
To construct a generic swallowtail along $\gamma$,
we choose $\mb b(u,v)$ so that
$\varDelta_1^S=\det(\xi(0),\xi'(0),\mb b(o))$ 
does not vanish, where $o:=(0,0)$.
Since 
$$
-\varDelta_0^S=\det(\xi(0),\xi'(0),\xi''(0)+2\mb b(o))=2\det(\xi(0),\xi'(0),2\mb b(o)),
$$
we have $\sigma_0^S\sigma_1^S<0$, proving the first part.
Moreover, in this situation, 
we can apply Theorem~\ref{prop:SW-rep}, and obtain a
generic swallowtail $f$ along $\gamma$ associated with this $\mb b$, 
which satisfies $\sigma_0^S\sigma^S<0$.
\end{proof}

\begin{Corollary}\label{Cor:966}
Let $\gamma$ be a non-generic space-cusp in $\R^3(0)$. 
Then there are no swallowtails along $\gamma$ whose
tail-part has negative Gaussian curvature.
\end{Corollary}

Since the tail-part is the side that the interior angle is equal to zero 
(cf. \cite[Chapter 6]{SUY2}), 
the condition that the Gaussian curvature is negative on the tail-part
is intrinsic. So, if $\gamma$ is non-generic, 
the possibility of the first fundamental from of
swallowtails is restricted.

\begin{proof}[Proof of Proposition~B]
It is sufficient to prove the assertion for $a=0$.
If $\gamma$ is a generic space-cusp, the existence of generic cusp along $\gamma$
is shown in Proposition~\ref{prop:937}.
Also, if $\gamma$ is a non-generic space-cusp, 
the existence of a generic cusp along $\gamma$
is shown in Proposition~\ref{prop:966}.
Thus, Proposition~B is proved.
\end{proof}

\begin{proof}[Proof of Corollary~C\,${}'$]
It is sufficient to prove the assertion for $a=0$.
Theorem C implies that
the singular set image of $f$ cannot be non-generic as a space-cusp
if $f$ is non-generic. So the first part of Corollary~C${}'$ follows.
As a consequence,
the singular set image of a non-generic swallowtail cannot 
lie in a totally geodesic 2-submanifold in $M^3(a)$.
The last assertion of Corollary~C${}'$ follows from
the existence of swallowtail
satisfying $\sigma_1^S(u)\equiv 0$
given in Proposition~\ref{prop:937}.
\end{proof}

\begin{Example}\label{eq:nonASs}
Consider the space-cusp given  by $\gamma(u):=(u^2,u^3,0)$.
Then $\gamma'(u)=u\xi(u)$, where $\xi(u)=(2,3u,0)$.
By setting $\mb b(u,v):=(0,0,1)$ and applying Theorem~\ref{prop:SW-rep}, 
we obtain a generic swallowtail
$$
f(u,v)=\gamma(u)+v\xi(u)+v^2\mb b(u,v)=
\left(u^2+2 v,u^3+3 u v,v^2\right),
$$
whose Gaussian curvature in $\R^3(0)$
is negative at the tail-part as in 
Figure~\ref{fig:Zero}, left.
\end{Example}

\begin{figure}[h!]
\begin{center}
\includegraphics[height=2.9cm]{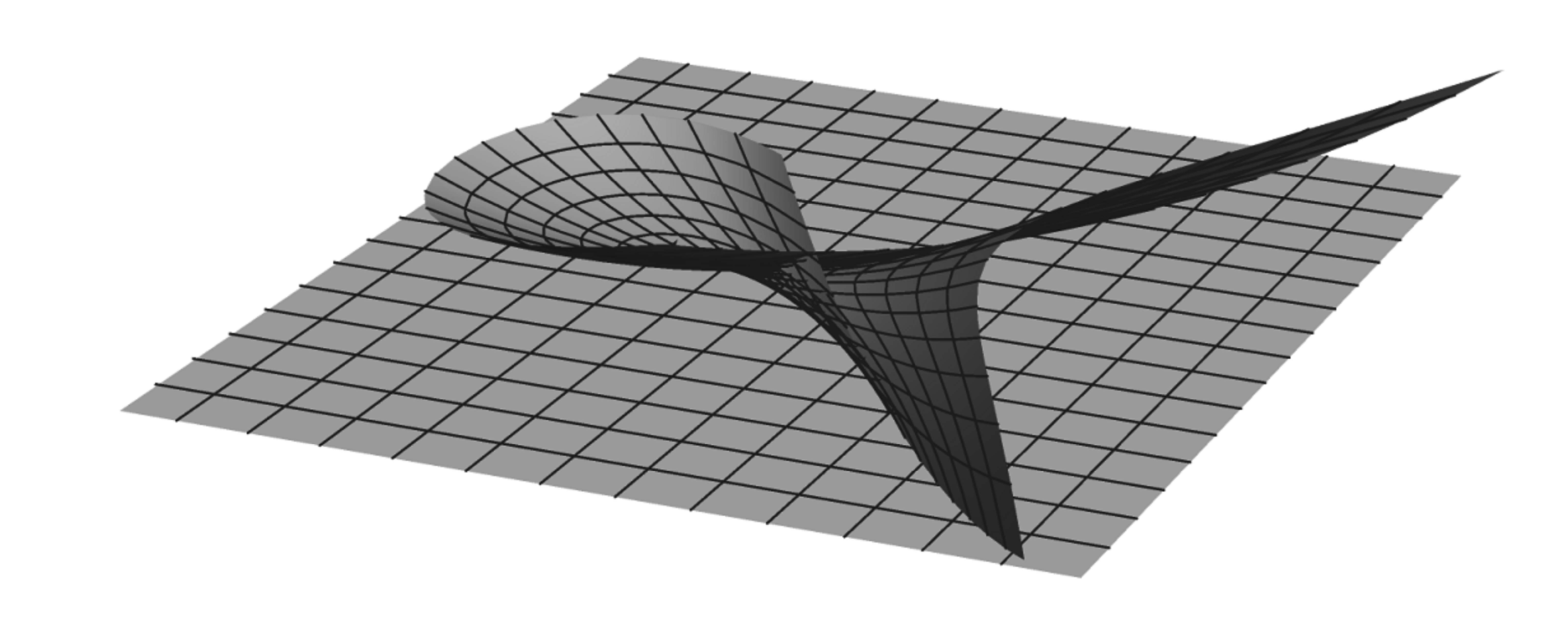}\,\,\,
\includegraphics[height=4.8cm]{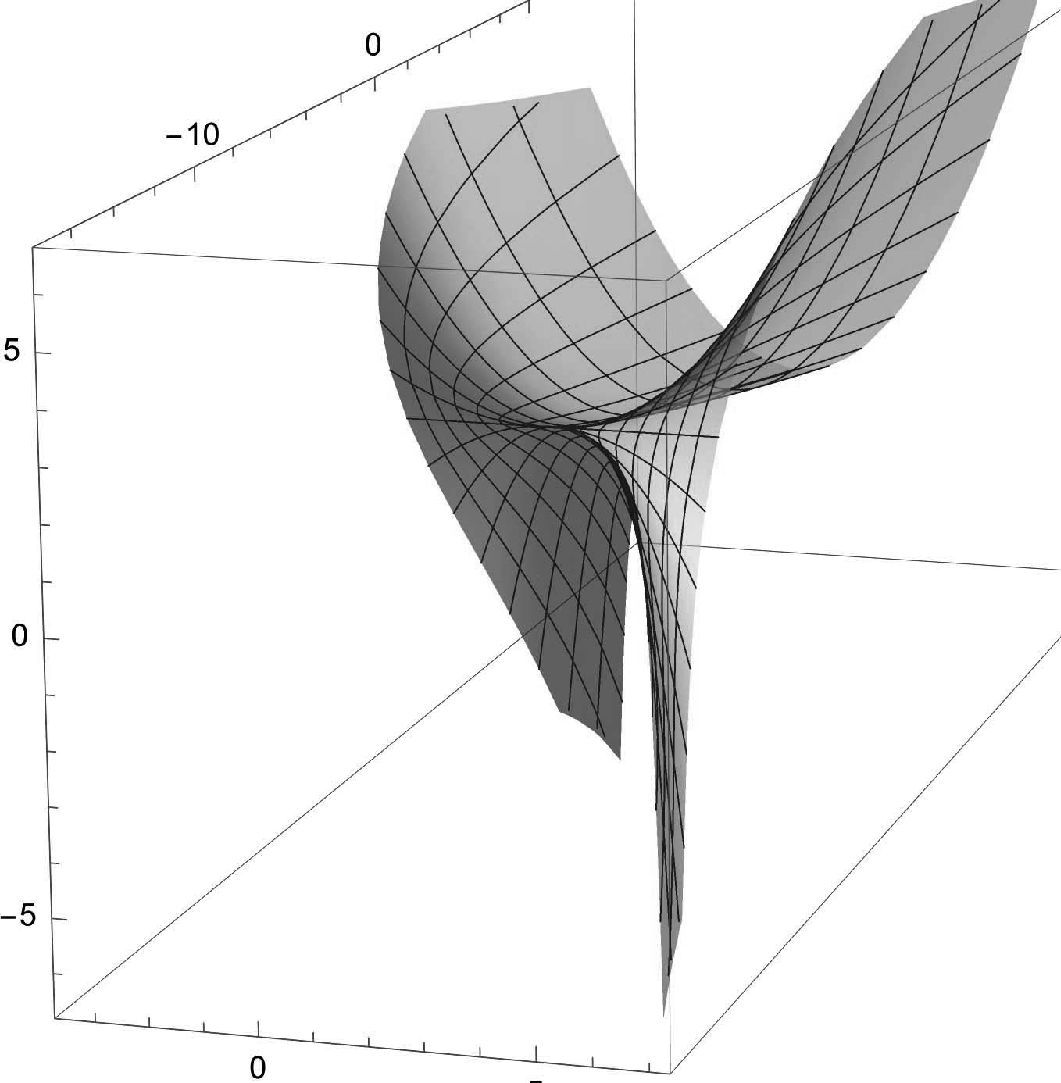}
\end{center}
\caption{
A generic swallowtail whose singular set image is planar
 (the plane is also indicated see the left) and
a generalized swallowtail with $\sigma^S=0$ whose singular set image is planar.
}
\label{fig:Zero}
\end{figure}

\begin{Example}\label{eq:gSflat}
Let $\gamma(u):=(u^2,u^3,0)$ 
be the space-cusp 
as in
Example \ref{eq:nonASs}. Then we can write
$\gamma'(u)=u\xi(u)$, where $\xi(u):=(2,3u,0)$.
By setting $\mb b(u,v):=(0,0,u)$ and applying Theorem~\ref{prop:SW-rep}, 
we obtain a non-generic generalized swallowtail
$$
f(u,v)=\gamma(u)+v\xi(u)+v^2\mb b(u,v)=
(u^2 + 2 v, u^3 + 3 u v, 2 u v^2)
$$
whose singular set is the image of $\gamma$.
Since $\gamma$ is not generic, $f$ does not give a swallowtail.
The surface has no self-intersections
(cf. Figure~\ref{fig:Zero}, right).
In fact, we 
suppose $f(u,v)=f(x,y)$ for 
$(u,v)\in \R^2\setminus \{(x,y)\}$.
The first component of the equality
\begin{equation}\label{eq:1405}
(u^2 + 2 v, u^3 + 3 u v, 2 u v^2)=(x^2 + 2 y, x^3 + 3 x y, 2 x y^2)
\end{equation}
implies $y=(-x^2-u^2+2v)/2$.
So, \eqref{eq:1405} reduces to
\begin{equation}\label{eq:1410}
\left(\frac{1}{2} (a-u) \left(a^2+a u-2 \left(u^2+3 v\right)\right),
u v^2-\frac{1}{4} a \left(-a^2+u^2+2 v\right)^2\right)=(0,0)
\end{equation}
Suppose that $a\ne u$. Then \eqref{eq:1410} yields
$v=\left(a^2+a u-2 u^2\right)/6$.
Substituting this into 
the second component of \eqref{eq:1410},
we obtain the equation
$$
0=u v^2-\frac{1}{4} x \left(-x^2+u^2+2 v\right)^2=
-\frac{1}{36} (x-u)^3 \left(4 x^2+7 x u+4 u^2\right).
$$
Since $u\ne x$, we have $4 x^2+7 x u+4 u^2=0$, which
implies $x=u=0$, a contradiction.
\end{Example}

We now prove Theorem~A in the introduction:

\begin{proof}[Proof of Theorem~A]
As pointed out at the end of Section~1, 
it is sufficient to prove the assertion for $\R^3(a)$ with $a=0$,
that is, we consider swallowtails in the Euclidean space $\R^3(0)$.
We may assume that $f_i$ $(i=1,2)$ is
induced by $\xi_i$ and $\mb b_i$ for $i=1,2$.
Without loss of generality,
we may assume that $\sigma_0^S>0$ for $f_1$ and $f_2$,
and also assume that $\xi_1$ and $\xi_2$ are  unit vector fields (cf. Remark \ref{rmk:unit}).
Since $\xi_i$ and $\xi'_i$ $(i=1,2)$
are linearly independent,
we can write
$$
\mb b_i(u)
=x_{i,1}(u,v)\xi_i(u)+x_{i,2}(u,v)\xi'_i(u)+x_{i,3}(u,v)\xi_i(u)\times \xi'_i(u)
\qquad (i=1,2),
$$
where $x_{i,j}(u,v)$ ($i=1,2,\,\, j=1,2,3$) are $C^\infty$-functions 
defined on a neighborhood of the origin in $\R^2$.
Since $f_i$ $(i=1,2)$ satisfies $\sigma^S>0$ 
$($resp. $\sigma^S<0)$, 
two functions $x_{1,3}(u,v)$ and $x_{2,3}(u,v)$ 
take the same sign as $\sigma^S$.
For each $t\in [0,1]$, we set
$$
\mb B^t_i(u):=
(1-t)\Big(x_{i,1}(u,v)\xi_i(u)+x_{i,2}(u,v)\xi'_i(u)\Big)+
x_{i,3}(u,v)\xi_i(u)\times \xi'_i(u).
$$
Since 
$$
\det(\xi_i,\xi'_i,\mb B^t_i)=\det(\xi_i,\xi'_i,\mb b_i),
$$
we obtain a germ  $h_i^t$ ($t\in [0,1]$)
of swallowtail 
associated with $\xi_i$ and $\mb B^t_i$ by 
Theorem~\ref{prop:SW-rep}.
In particular, $f_i$ can be deformed to 
the germ of  swallowtail $h^1_i$
associated with $\xi_i$ and $x_{i,3}(u,v)\xi_i(u)\times \xi'_i(u)$.

Let $\mb c_i(u)$ ($i=1,2$) be
the space curve satisfying $\mb c_i(0)=\mb 0$ and
$\mb c'_i(u)=\xi_i(u)$.
Then $u$ is the arc-length parameter of $\mb c_i$. 
Since $\xi_i(u)\times \xi'_i(u)$ does not vanish, $\mb c_i(u)$ is a regular space curve
whose curvature function $\kappa_i(u)$ is positive.
Moreover, since $\sigma_0^S>0$, the torsion function $\tau_i(u)$ of $\mb c_i(u)$
is also positive.
Then it holds that 
(cf. \cite[(5.7)]{UY}),
$$
|\xi'_i(u)|=\kappa_i(u),\qquad \det(\xi_i(u),\xi'_i(u),\xi''_i(u))=\kappa_i(u)^2\tau_i(u)
\qquad (i=1,2).
$$
Regarding this, we set
$$
K^t(u):=(1-t)\kappa_1(u)+t\kappa_2(u),\quad
T^t(u):=\frac{(1-t)\kappa_1(u)^2\tau_1(u)+t\kappa_2(u)^2\tau_2(u)}{K^t(u)^2}.
$$
By the fundamental theorem for space curves
(cf. \cite[Theorem~5.2]{UY}),
there exists a smooth family of space curves $\Gamma^t(u)$ ($t\in [0,1]$)
satisfying $\Gamma^0=\mb c_1$ and $\Gamma^1=\mb c_2$
whose curvature and torsion are $K^t$ and $T^t$ respectively, and
$u$ gives the arc-length parameter of $\Gamma^t$. 
One can easily check that the data
$$
\xi^t(u):=\frac{d\Gamma^t(u)}{du},\qquad 
\check {\mb B}^t(u):=\Big((1-t)x_{1,3}(u,v)+t x_{2,3}(u,v)\Big)\xi^t(u)\times (\xi^t)'(u)
$$
induce a 1-parameter  family of germs of swallowtails $k^t(u,v)$ ($t\in [0,1]$).
Since $x_{1,3}(u,v)$ and $x_{2,3}(u,v)$ are both positive-valued or
both negative-valued, each $k^t$ gives a generic swallowtail.
In fact, we have 
\begin{align*}
\det(\xi^t,(\xi^t)',(\xi^t)'')&=(K^t)^2T^t=(1-t)\kappa_1\tau_1+t\kappa_2\tau_2\\
&=(1-t)\det(\xi_1,\xi'_1,\xi''_1)+t\det(\xi_2,\xi'_2,\xi''_2)(>0)
\end{align*}
and
$$
\det(\xi^t,(\xi^t)',\mb b^t)=\Big((1-t)x_{1,3}+t x_{2,3}\Big)|\xi^t\times (\xi^t)'|^2.
$$
Since $k^0=h^1_0$ and $k^1=h^1_2$, we obtain the deformation of
swallowtails from $h^1_1$ to $h^1_2$ preserving the genericity.
Since we have already shown that each $f_i$ ($i=1,2$) can be deformed
to $h^1_i$, proving the assertion.
\end{proof}

If we remove the genericity of
swallowtail, we can prove the following:

\begin{Proposition}
\label{prop:s686}
Let $f_i$ $(i=1,2)$ be
two germs of swallowtails
in $\mathcal F^S(\R^3(a))$.
Then $f_1$ can be deformed to $f_2$ in
the class of $\mathcal F^S(\R^3(a))$.
\end{Proposition}

To prove this, we may set $a=0$.
We prepare the following two lemmas:

\begin{Lemma}
\label{prop:s686a}
Let $f$ be a germ of generic swallowtail with $\sigma_0^S>0$ and 
$\sigma^S<0$.
Then there exists a 
diffeomorphism
germ $\phi$  at the origin $o$ of $\R^2$
such that $f\circ \phi$
can be deformed to 
a germ of generic swallowtail 
with $\sigma_0^S>0$ and $\sigma^S>0$.
\end{Lemma}

\begin{proof}
We may assume that $f$ is associated with $\xi$ and $\mb b$.
Since $f$ is generic, we have $\phi:=\det(\xi,\xi',\mb b)< 0$.
We set $\psi:=\det(\xi,\xi',\xi'')$.
Since $\sigma_0^S>0$, we have  $\psi+2\phi>0$.
We set
$$
\mb B^t(u,v):=t \mb b(u,v) \qquad (t\in [-1,1]).
$$
Since $\psi+2t \phi>0$ for each $t\in [-1,1]$,
we can apply Theorem~\ref{prop:SW-rep},
and obtain 
the 1-parameter family of 
swallowtails associated with $\xi$ and
$\mb B^t$ ($t\in [-1,1]$), which 
give the deformation of
 $f$ with $\sigma^S<0$ 
to the swallowtail with $\sigma^S>0$.
\end{proof}

\begin{Lemma}
\label{prop:s686b}
Let $f$ be a germ of swallowtail with $\sigma_0^S>0$
which is not generic.
Then there exists an orientation preserving 
diffeomorphism
germ $\phi$ of $\R^2$ at the origin
such that $f\circ \phi$
can be deformed to 
a germ of generic swallowtail with $\sigma_0^S>0$ and $\sigma^S> 0$.
\end{Lemma}

\begin{proof}
We may assume that $f$ is associated with $\xi$ and $\mb b$.
Since $f$ is not generic (i.e. $\sigma^S=0$), 
\eqref{eq:1232} holds, that is
\begin{equation}\label{eq:sigma894}
0<\sigma_0^S=\det(\xi(0),\xi'(0),\xi''(0)).
\end{equation}
We set
$$
\mb B^t(u):=(1-t)\mb b(u)+t \xi'' \qquad (t\in [0,1]).
$$
By \eqref{eq:sigma894},
we can apply Theorem~\ref{prop:SW-rep}, and obtain 
the 1-parameter family of 
swallowtail associated with $\xi$ and
$\mb B^t$, which gives 
a deformation of $f$ to the swallowtail satisfying
$$
0<\sigma_0^S=\det(\xi(0),\xi'(0),\mb B^1(0))=\det(\xi(0),\xi'(0),\xi''(0))
=\sigma^S.
$$
\end{proof}

\begin{proof}[Proof of Proposition~\ref{prop:s686}]
By Lemma~\ref{prop:s686b}, we may assume that 
$f_i$ ($i=1,2$) satisfies $\sigma_0^S>0$ 
and $\sigma^S\ne 0$.
Moreover, by Lemma~\ref{prop:s686a}, we may also assume that 
$f_i$ ($i=1,2$) has positive $\sigma^S$.
So we can apply Theorem~A.
\end{proof}

\section{Asymptotic swallowtails}

In this section, we prove Theorems~D and Corollary D${}'$ stated in the
introduction.
We let $f$ be a swallowtail as in \eqref{eq:f265}.
We assume that $f$ is asymptotic and may set $a=0$.
Since $(\mb b(u,v)-\mb b(u,0))/v$ is an $\R^3$-valued  smooth function,
Corollary~\ref{cor:682} yields that
$\mb b(u,v)$ can be written
in the following form
\begin{equation}\label{eq:B_asym}
\mb b(u,v)=p(u)\xi(u)+q(u)\xi'(u)+v {\mb r}_0(u,v),
\end{equation}
where $p(u)$ and $q(u)$ are $C^\infty$-function germs 
at $u=0$ and $\mb r_0(u,v)$ is an $\R^3$-valued
 $C^\infty$-function germ
at $o$. Then we have
\begin{equation}\label{eq:f1222}
f(u,v)=\gamma(u)+(v+v^2p(u))\xi(u)+v^2 q(u)\xi'(u)
+v^3{\mb r}_0(u,v).
\end{equation}
Replacing $v$ by $v+v^2p(u)$, we have the following
expression
\begin{equation}\label{eq:f652}
f(u,v)=\gamma(u)+v \xi(u)+v^2 q(u)\xi'(u)+v^3
{\mb r}(u,v),
\end{equation}
where $\mb r(u,v)$ is a certain $C^\infty$-function defined on
a neighborhood of the origin in $\R^2$.
Then we have
\begin{align}
f_u(u,v)&=u \xi(u)+v\xi'(u)+O(v^2),\\ 
f_v(u,v)&=\xi(u)+2v q(u) \xi'(u)+
3v^2\mb r(u,v)+ O(v^3).
\end{align}
So, we have
\begin{align*}
f_{uu}(u,v)&=\xi(u)+u\xi'(u)+v\xi''(u)+O(v^2), \\
f_{uv}(u,v)&=\xi'(u)+2v\Big(q'(u)\xi'(u)+q(u)\xi''(u)\Big)+O(v^2), \\
f_{vv}(u,v)&=2q(u) \xi'(u)+6v{\mb r}(u,v)+O(v^2).
\end{align*}
Since 
$$
f_u(u,v)\times f_v(u,v)
=v\Big(2u q(u)-1\Big)\xi(u)\times \xi'(u) +O(v^2),
$$
we can  write
\begin{equation}\label{nu:686}
\nu(u,v)=\frac{\xi(u)\times \xi'(u)}{2u q(u)-1}+O(v).
\end{equation}
We let
$$
L:=f_{uu}\cdot \nu,\quad M:=f_{uv}\cdot \nu,\quad N:=f_{vv}\cdot \nu
$$
be the coefficients of the second fundamental form of $f$.
Then, we have
\begin{align*}
\lambda&=\det(f_u(u,v),f_v(u,v),\nu(u,v))\\
&=(f_u(u,v)\times f_v(u,v)) \cdot \nu(u,v)\\
&=v |\xi(u)\times \xi'(u)|^2 +O(v^2)
\end{align*}
and
\begin{align*}
\lambda L&=(f_u\times f_v)\cdot f_{uu}=v^2(2uq-1)\det(\xi,\xi',\xi'')+O(v^3), \\
\lambda M&=(f_u\times f_v)\cdot f_{uv}=2v^2q(2uq-1)\det(\xi,\xi',\xi'')+O(v^3),\\
\lambda N&=(f_u\times f_v)\cdot f_{vv}=6 (2uq-1) v^2 \det(\xi,\xi',\mb r)+O(v^3).
\end{align*}
The Gaussian curvature of $f$ in $\R^3(0)$ satisfies
\begin{equation}\label{eq:K700}
K=\frac{LN-M^2}{\lambda^2}
=\frac{2v^4}{\lambda^4}\det(\xi(u),\xi'(u),\xi''(u))
\Big(\Delta_{q,\mb r}(u,v)+O(v)\Big),
\end{equation}
where
\begin{align}\label{eq:Dqr1229}
\Delta_{q,\mb r}(u,v) 
&:=
\Big(2u q(u)-1\Big)^2
\Big(3 \det(\xi(u),\xi'(u),{\mb r}(u,v)) \\
&\phantom{aaaaaaaaaaaaaaaaaaaaaaa} \nonumber
-2q(u)^2\det(\xi(u),\xi'(u),\xi''(u))\Big). 
\end{align}
Since 
$$\dy\lim_{v\to 0}\frac{\lambda(u,v)}v=\lambda_v(u,0)\ne 0,
$$
we have $\det(\xi,\xi',\xi'')\ne 0$ when $f$ is positively curved
or negatively curved.
Moreover, $\dy\lim_{v\to 0}K(u,v)$ is positive
(resp. negative) if and only if 
so is $\Delta_{q,\mb r}(u,0)$.

We now give a formula that produces 
all of swallowtails in $\R^3(a)$ whose extrinsic Gaussian curvature are 
positive or negative everywhere.
When $a=0$, such examples have been known.
In fact the Kuen surface (\cite[Example 9.6]{UY}) is a 
surface of constant negative curvature
and has a swallowtail singular points.
Akamine \cite{A} gave concrete examples of
positively curved swallowtails in $\R^3(0)$.

\begin{Theorem}\label{thm:main}
Let $\gamma(u)$ be a generic space-cusp in $\R^3(a)$
having the expression $\gamma'(u)=u\xi(u)$.
Let $q(u)$ be a $C^\infty$-function germ at $u=0$
and $\mb r(u,v)$ an $\R^3$-valued
$C^\infty$-function germ at $u=0$.
If $\Delta_{q,\mb r}(o)$ given by \eqref{eq:Dqr1229}
is positive $($resp. negative$)$,
then the $C^\infty$-map $f$
defined by \eqref{eq:f652}
gives a germ of positively $($resp. negatively$)$ curved
swallowtail at $o$ in $\R^3(a)$.
Conversely, any germs of positively $($resp. negatively$)$ curved
swallowtails at $o$ in $\R^3(a)$ can be obtained in this manner.
\end{Theorem}

\begin{proof}
By Proposition \ref{prop:903}, it is sufficient to prove in the case of $a=0$.
We let $f$ be a germ of swallowtail 
as in \eqref{eq:f652}
associated with the data $\xi$, $q$ and $\mb r$.
By \eqref{nu:686},
we have
\begin{equation}\label{eq:nu_u}
\nu_u=\left.\left(\frac{1}{2uq-1}\right)_{\!\!u}\right|_{u=0} \xi(0)\times \xi'(0)
- \xi(0)\times \xi''(0).
\end{equation}
Since $\gamma$ is a generic space-cusp,
we have $\det(\xi,\xi',\xi'')\ne 0$.
In particular,
$\xi(0)\times \xi''(0)$ does not vanish,
and $\{\xi(0)\times \xi'(0),\,  \xi(0)\times \xi''(0)\}$ are linearly independent.
So, $\nu_u(o)\ne \mb 0$ and
 $f$ is a wave front. By Theorem~\ref{prop:SW-rep}, 
we can conclude that $f$ is a germ of 
swallowtail. Moreover, if $\Delta_{q,\mb r}(o)$ 
is positive (resp.~negative),
$f$ is positively (resp. negatively) curved. 

The converse assertion follows from
the discussions before 
Theorem~\ref{thm:main}.
\end{proof}

\begin{Remark}
Since generic space-cusps are
positive or negative,
germs of positively $($resp. negatively$)$ curved
swallowtails  
are also divided into two classes.
However, the two classes can be interchanged.
In fact, if the singular curve
$\gamma(u):=f(u,0)$
of  $f(u,v)$ is a positive (resp. negative) space-cusp, then 
the singular curve
of $-f(u,v)$ and $f(-u,-v)$ are both negative (resp. positive) space-cusps. 
\end{Remark}

\begin{figure}[h!]
\begin{center}
\includegraphics[height=2.3cm]{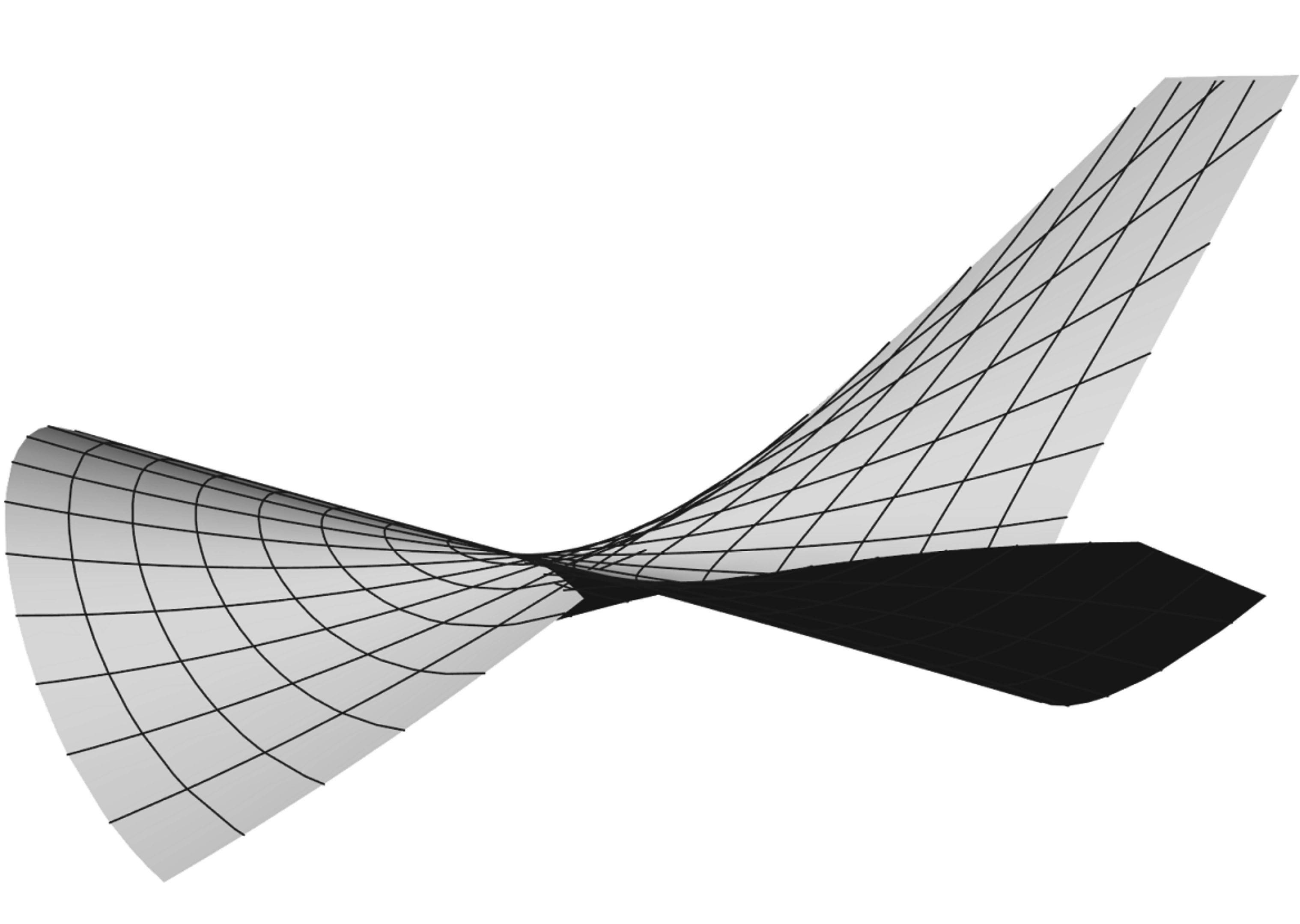}\qquad 
\includegraphics[height=1.8cm]{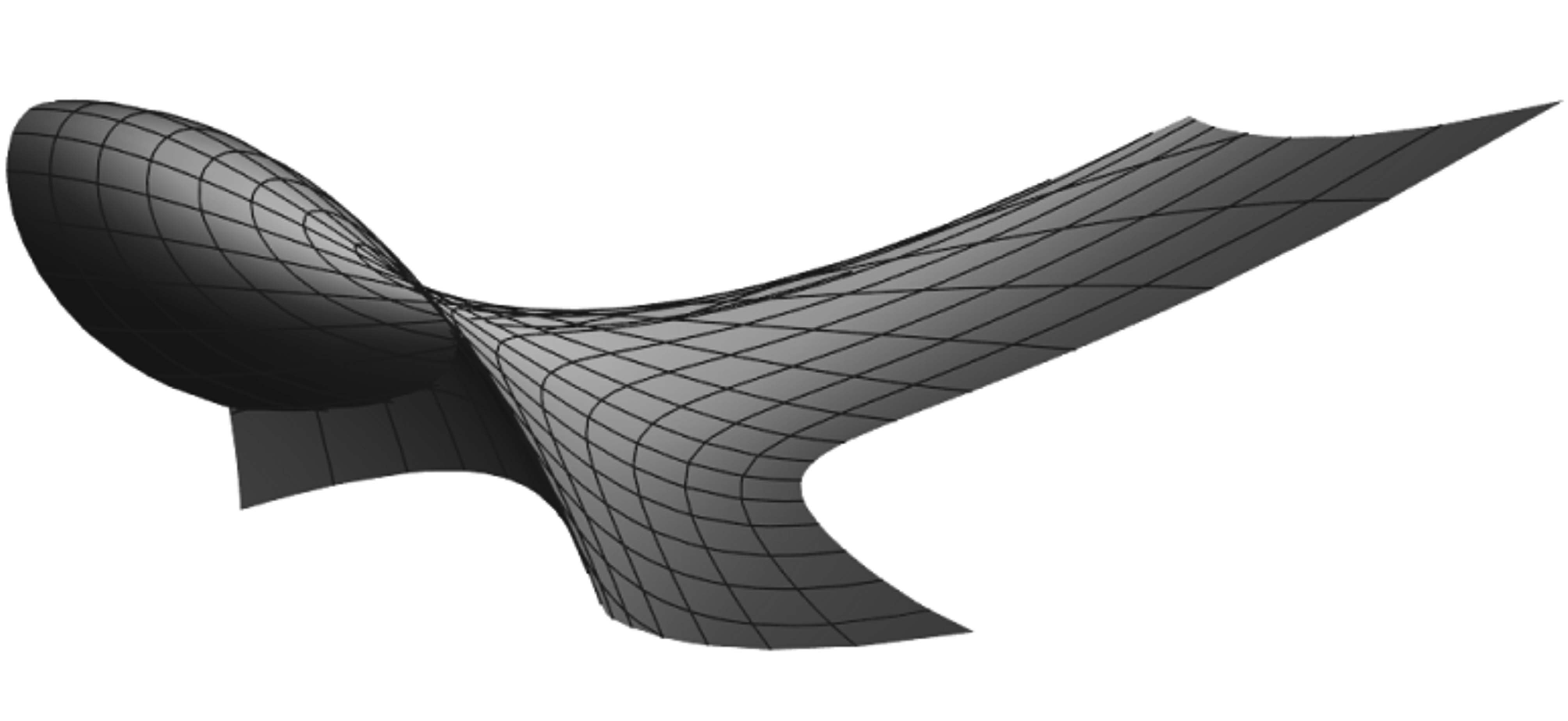}\quad
\includegraphics[height=2.3cm]{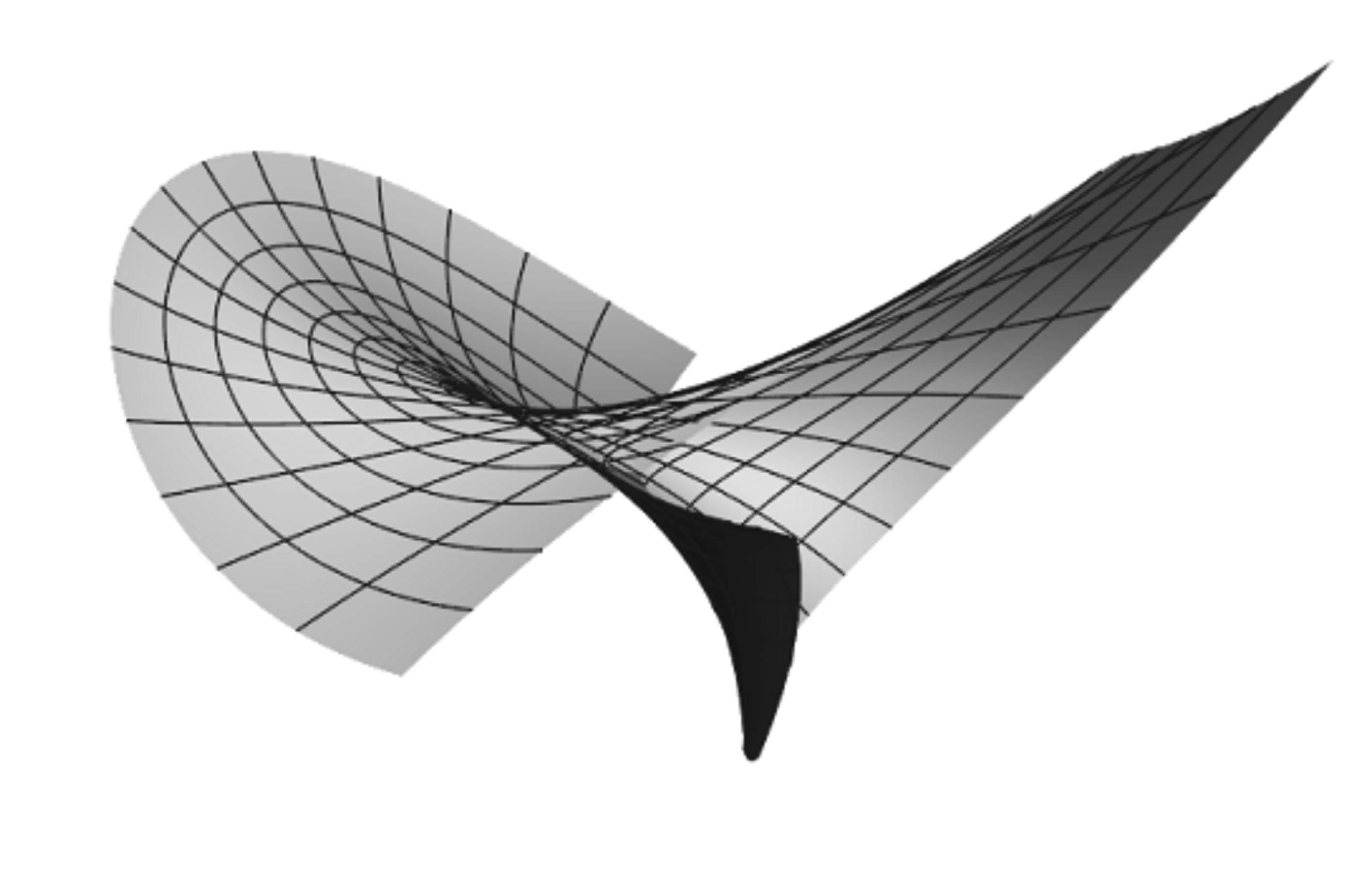}
\end{center}
\caption{
The swallowtail  (left) in 
Example \ref{1122a}, and
the positively curved
swallowtail $f_+$ (center) and
the negatively curved
swallowtail $f_-$ (right)
along the same space-cusp
given in Example \ref{1122b}.
}
\label{fig:PN}
\end{figure}

As an application of Theorem~\ref{thm:main},
we give examples of positively curved or 
negatively positively curved swallowtails.

\begin{Example}[Asymptotic swallowtails of parabolic type]\label{1122a}
Let $\mb r$ be identically zero
in the expression \eqref{eq:f652}.
We then consider a swallowtail
\begin{equation}\label{eq:f1216}
f(u,v)=\gamma(u)+v \xi(u)+v^2 q(u)\xi'(u),
\end{equation}
in $\R^3(0)$ which has the Gaussian curvature
$$
K=\frac{LN-M^2}{\lambda^2}
=-\frac{2v^4}{3\lambda^4}\det(\xi,\xi',\xi'')^2q(u)^2
\Big(2uq(u)-1\Big)^2
(\le 0).
$$
Since $f$ is foliated by parabolas, we call such an $f$
asymptotic swallowtail of {\it parabolic type}.
As long as $\det(\xi,\xi',\xi'')$ and
$q$ have no zeros, $f$ gives a negatively curved swallowtail.
For example, if
$\xi(u):=(1,u,u^2)$, we have
$\gamma(u)=(u^2/2,u^3/3,u^4/4)$.
The case that $q(u):=1/10$ is shown in Figure~\ref{fig:PN}, left.
It is interesting that if we set $q(u):=1$, then we have
$$
f(u,v)=\left(\frac{u^2}{2}+v,\frac{u^3}{3}+u v+v^2,\frac{u^4}{4}+u^2 v+2 u v^2\right).
$$
Although $o$ is a swallowtail singular point,
$f([-2/3,2/3]\times [-1/3,1/3])$ as 
in Figure~\ref{fig:beq1} (left) looks like that $f$ has no self-intersections.
This is because there exists another non-cuspidal edge singular point 
at $(u,v)=(1/2,0)$ and the set of self-intersections of $f$ is compact.
By setting $v=w-u^2/2$, $f$ is written as
$$
f(u,w)=\left(w,\frac{u^4}{4}-\frac{u^3}{6}-u^2 w+u w+w^2,
\frac{1}{4} u \left(2 u^4-u^3-8 u^2 w+4 u w+8 w^2\right)\right).
$$
Figure~\ref{fig:beq1} (right)
corresponds to the image $f([-0.007,0.01]\times [-0.26,0.35])$
and readers can recognize the shape of swallowtail.
\end{Example}

\begin{figure}[h!]
\begin{center}
\includegraphics[height=3.3cm]{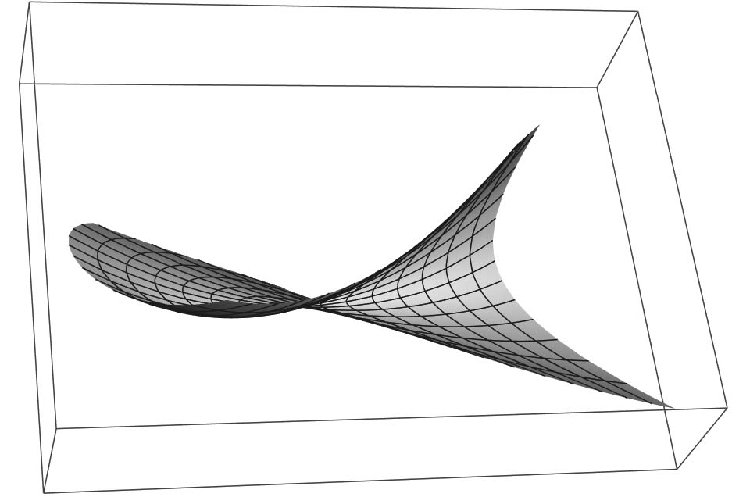}\qquad\qquad 
\includegraphics[height=2.5cm]{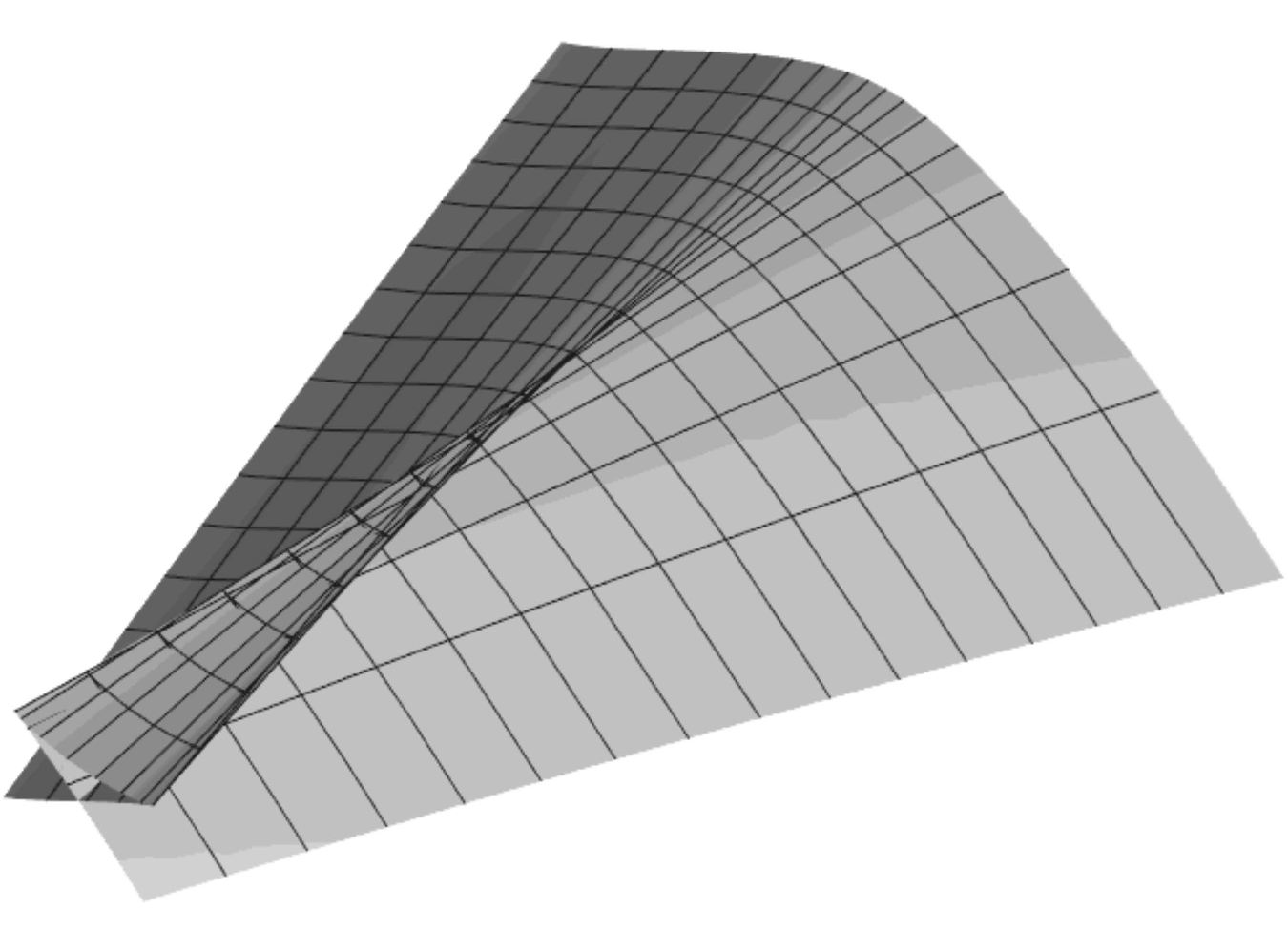}
\end{center}
\caption{
The image of $f$ for $q=1$
in Example \ref{1122a} (left) and
its enlarged view (right)
}
\label{fig:beq1}
\end{figure}

If we set $q$ is identically equal to zero,
we obtain an asymptotic swallowtail with zero Gaussian curvature:

\begin{Proposition}\label{Prop:2042}
For any generic space-cusp $\gamma$ in $\R^3(0)$, 
there exists 
an asymptotic swallowtail with zero Gaussian curvature
as a tangential developable of $\gamma$ in $R^3(0)$.
Moreover, such an asymptotic swallowtail 
has zero extrinsic Gaussian curvature 
in $\R^3(a)$ for  each $a\in \R$.
\end{Proposition}

Let $f$ be an asymptotic swallowtail $f$ with zero Gaussian curvature in $\R^3(0)$.
The extrinsic Gaussian curvature $K_{ext}(a)$ of $f$ 
(which is considered to lie in $\R^3(a)$)
is given by $K_{a}-a$, where
$K_a$ is the the Gaussian curvature $K_{a}$ of $f$ in $\R^3(a)$.
By \eqref{eq:908},
$K_{ext}(a)$ vanishes identically, since $K_0$ vanishes.
So the last statement of
Proposition \ref{Prop:2042} is proved.
We give examples of non-parabolic type:

\begin{Example}\label{1122b}
We consider the case that $q=0$ and
$
\mb r:=\pm \xi\times \xi'
$,
then
$$
f_\pm(u,v)=\gamma(u)+v\xi(u)\pm v^3 \xi(u)\times \xi'(u)
$$
gives a germ of swallowtail in $\R^3(0)$ such that
$
\Delta_{q,\mb r}
=\pm 6 |\xi\times \xi'|^2.
$
So $f_+$ (resp.~$f_-$)
is a positively (resp. negatively)
curved.
For example, if $\xi(u):=(1,u,u^2)$, then
we have
(see Figure~\ref{fig:PN}, center and right)
$$
f_\pm(u,v)=\left(\frac{u^2}2+v\pm u^2v^3,\frac{u^3}3+uv\mp 2uv^3, \frac{u^4}4+u^2v\mp v^3 \right).
$$
\end{Example}

To prove Theorem~D, we prepare the following:

\begin{Lemma}\label{prop:sw1248}
Let $f$ be a germ of positively $($resp. negatively$)$
curved swallowtail with $\sigma_0^S>0$ in $\R^3(a)$, 
which is associated with the data $(\xi, q,\mb r)$.
If $\Delta_{q,\mb r}(o)$ is positive $($resp. negative$)$,
then $f$ can be deformed
to the germ of swallowtail associated with
the data $(\xi, 0,\xi\times \xi')$  $($resp. $(\xi, 0,-\xi\times \xi'))$
in $\R^3(a)$.
\end{Lemma}

\begin{proof}
Without loss of generality, we may assume 
that $a=0$ (cf. Proposition~\ref{prop:903}).
We may assume that $f$ is induced by data $(\xi,q,\mb r)$.
We first consider the case that
$\Delta_{q,\mb r}(o)$ is positive.
\begin{itemize}
\item We consider the case $\det(\xi(0),\xi'(0),\mb r(o))>0$.
Since $\sigma_0^S>0$, we have $\det(\xi,\xi',\xi'')>0$.
We set
$$
q^t:=t q,\qquad \mb r^t:=(1-t)\mb r+t \xi\times \xi'
\qquad (t\in [0,1]),
$$
then
$$
\Delta_{q^t,\mb r^t}=
\biggl(
(1-t)\det(\xi,\xi',\mb r)+(t+t^2q^2)\det(\xi,\xi',\xi'')
\biggr)
\Big(2uq(u,v)-1\Big)^2>0
$$
and $(\xi,q^t,\mb r^t)$
induces a deformation of $f$
to the swallowtail associated 
with the data $(\xi,0,\xi\times \xi')$.

\item We next consider the case $\det(\xi(0),\xi'(0),\mb r(o))<0$.
Since $\Delta_{q,\mb r}(o)>0$, 
we may assume
$|\det(\xi,\xi',\mb r)|<q^2 \det(\xi,\xi',\xi'')$.
Then the data
$$
q^t:=q,\qquad \mb r^t:=(1-t)\mb r
\qquad (t\in [0,1])
$$
satisfy
$$
\Delta_{q^t,\mb r^t}=\Delta_{q^t,\mb r}-t\det(\xi,\xi',\mb r)>0.
$$
So $(\xi,q^t,\mb r^t)$ induces a 1-parameter family of swallowtails,
by which $f$ can be deformed to the swallowtail $h$
associated with the $(\xi,q,\mb 0)$.
Then, the data
$$
q^t:=(1-t)q,\qquad \mb r^t:=t \xi\times \xi'
\qquad (t\in [0,1]),
$$
gives the deformation of $h$ to the swallowtail
associated with $(\xi,0,\xi\times \xi')$.
\end{itemize}
We next consider the case that
$\Delta_{q,\mb r}(o)$ is negative.
Since $\sigma_0^S>0$, we have $\det(\xi,\xi',\xi'')>0$.
In this case, we have $\det(\xi,\xi',\mb r)<0$, and
the data
$$
q^t:=\sqrt{1-t} q,\qquad \mb r^t=(1-t)\mb r-t \xi\times \xi'
\qquad (t\in [0,1])
$$
satisfy
$$
\Delta_{q^t,\mb r^t}=(1-t)\Delta_{q^t,\mb r}-t|\xi\times \xi'|^2(<0).
$$
So $(\xi,q^t,\mb r^t)$ induces a 1-parameter family of swallowtails,
by which $f$ can be deformed to the swallowtail
associated with the data $(\xi,0,-\xi\times \xi')$.
\end{proof}

We now prove Theorem~D in the introduction.

\begin{proof}[Proof of Theorem~D]
As pointed out at the end of Section~1, 
it is sufficient to prove the assertion 
for swallowtails lying in $\R^3(0)$.
Let $f$ be a germ of an asymptotic swallowtail written
as in \eqref{eq:f1222}.
Then, by replacing $q(u,v)$ and $\mb r(u,v)$ by
$$
t q(u,v),\qquad t\mb r(u,v)\qquad (t\in [0,1])
$$
and we can give a deformation of $f$, by which
we may assume that $f_i$ can be written as 
$$
f_i(u,v)=\gamma_i(u)+v\xi_i(u),
$$
where $\gamma'_i=u\xi_i(u)$.
By Remark \ref{rmk:unit}, we may assume that
$\xi_i$ ($i=1,2$) are unit vector fields along $\gamma$.
Then as in the proof of Theorem~A,
one can give a continuous deformation $\hat \xi^t$ ($t\in [0,1]$)
of $\xi_1$ to $\xi_2$ (i.e. $\hat \xi^0=\xi_1$ and $\hat \xi^1=\xi_2$)
so that $\gamma^t(u):=\int_0^u w \hat \xi^t(w)dw$ 
gives a generic space-cusp for each $t$
as generic space-cusps.
Then, by a 1-parameter family of the data
$(\hat \xi^t,0,0)$ gives a
a deformation of
$f_1$ to $f_2$.
So we obtain the first assertion.

We next consider $f_1$ (resp. $f_2$) which is a positively (resp. negatively)
curved swallowtail. 
Without loss of generality, we may assume that they
have positive $\sigma^S_0$.
We may assume that $f_i$ ($i=1,2$) is
associated with the data $(\xi_i, q_i,\mb r_i)$.
Since $\sigma^S_0>0$, we have that 
$$
\det(\xi_i(0),\xi'_i(0),\xi''_i(0))>0.
$$
By Lemma~\ref{prop:sw1248},
each $f_i$ ($i=1,2$) can be 
deformed to the swallowtail $\hat f_i$
associated with the data $(\xi,0,\xi_i\times \xi'_i)$
(resp. $(\xi,0,-\xi_i\times \xi'_i)$).
Like as in the above discussion, we may assume that
$\xi_i$ ($i=1,2$) is a unit vector field,
and can give a continuous deformation $\hat \xi^t$ ($s\in [1,2]$)
of $\xi_1$ to $\xi_2$ (i.e. $\hat \xi^0=\xi_1$ and $\hat \xi^1=\xi_2$)
so that $\hat\gamma^t(u):=\int_0^u w \hat \xi^t(w)dw$ 
gives a generic space-cusp for each $t$
as generic space-cusps.
Then, by the following 1-parameter family of the data
$$
\hat \xi^t,\quad
(q,\mb r):=(0, \hat \xi^t\times \hat \xi^t)
\quad (\text{resp. } (q,\mb r):=(0, \hat \xi^t\times \hat \xi^t))
\qquad (t\in [0,1]),
$$
we obtain a deformation of
$\hat f_1$ to $\hat f_2$.
\end{proof}

We then prove Corollary D${}'$ in the introduction:

\begin{proof}[Proof of Corollary D${}'$]
We first remark that the asymptoticity of swallowtails in $\R^3(a)$
does not depend on $a\in \R$
(cf. Proposition \ref{prop:903}).
We first set $a=0$.
We let $f:U\to \R^3(0)$ be a swallowtail 
whose Gaussian curvature
is constant $c_0\in \{1, -1\}$.
If $c_0=-1$, we can take $f$ as a Kuen surface 
(\cite[Example 9.6]{UY}).
If $c_0=1$, there exists a swallowtail of
constant curvature $1$ by Appendix~\ref{App1}.

Then, by applying \cite[Proposition~3.2]{SUY2022},
we obtain 
a swallowtail of
positive constant extrinsic curvature and 
a swallowtail of
negative constant extrinsic curvature and in $\R^3(a)$ for any $a\in \R$.
We denote it by $\tilde f:U\to M^3(a)$.
By setting $f_1:=f$ and
$f_2:=\tilde f$, we obtain the desired deformation by applying
Theorem~D. So Corollary~D${}'$  is proved.
\end{proof}

\appendix
\section{Generalized cuspidal edges and the 
invariants $\sigma^C_0$ and $\sigma^C_g$}
\label{app:0}

We fix a Riemannian $3$-manifold $M^3$
with Riemannian metric $g$.

\begin{Definition}
Let $f:U\to M^3$ be a
a frontal with admissible parametrization  (cf. Definition~\ref{def:187}).
A singular point $p\in U$ of the first kind
is called a {\it generalized cuspidal edge}.
\end{Definition}

When $M^3$ is the Euclidean 3-space, this concept 
coincides with that in \cite{HNUY}, \cite{HNSUY} and \cite{FKPUY}.
By Fact \ref{fact544}, 
we have following:

\begin{Prop}\label{prop2072}
Let $f:U\to M^3$ be a generalized cuspidal edge at $o$. 
Then $o$ is a cuspidal edge singular point if and only if
$f$ is a wave front at $o$.
\end{Prop}

We denote by $\mc F^{GC}(M^3)$ the set of germs of
generalized cuspidal edges with admissible parametrization.
By definition, we have
$$
\mc F^C(M^3)\subset \mc F^{GC}(M^3).
$$
For  a generalized swallowtail $f:U\to M^3$ 
belonging to $\mc F^{GC}(M^3)$,
we set (cf. \cite[(1.1)]{SUY2022})
\begin{equation}\label{eq:sigma0C}
 \sigma_{0}^C(u):=\op{sgn}
  \Big(\det(f_u(u,0),\nabla_vf_{v}(u,0),
\nabla_v\nabla_vf_{v}(u,0))\Big).
\end{equation}
The following assertion holds:

\begin{Prop}\label{prop:2086}
Let $f:U\to M^3$ be a generalized 
cuspidal edge at $o$ belonging to
$\mc F^{GC}(M^3)$.
Then $o$ is a cuspidal edge singular point if and only
$\sigma_{0}^C$ does not vanish.
\end{Prop}

\begin{proof}
By \cite[(3.4)]{MSUY}, $\sigma_{0}^C(u)$ is just the numerator of
the cuspidal curvature $\kappa_C$ along the $u$-axis.
By \cite[Proposition 3.11]{MSUY}, 
$\kappa_C$ at $o$ does not
vanish if and only if $f$ is a wave front at $o$.
By Fact \ref{fact544}, 
$\kappa_C$ at $o$ does not
vanish if and only if $o$ is a
cuspidal edge singular point.
So we obtain the conclusion.
\end{proof}

\begin{Cor}
 For each choice of $f\in \mc F^{GC}(M^3)$, 
 the function $\sigma_0^C(u)$ 
 and does not
 depend on the choice of 
 \begin{itemize}
  \item a Riemannian metric of $M^3$,
  \item an admissible coordinate system $(U; u,v)$, and
  \item an orientation-compatible local coordinate system $\Phi$ of $M^3$. 
 \end{itemize}
\end{Cor}

\begin{proof}
By Proposition \ref{prop:2086}, it sufficient to consider the case that
$\sigma_0^C(u)\ne 0$. Then $o$ is a cuspidal edge singular point,
and so the assertion has been proved in \cite[Proposition 1.1]{SUY2022}.
\end{proof}

Let $f:U\to M^3$ be a generalized 
cuspidal edge at $o$ belonging to
$\mc F^{GC}(M^3)$.
We set (cf. \cite[(0.5)]{SUY2022})
\begin{equation}\label{eq:sigma0a}
 \sigma^C_{g}(u):=\op{sgn}
\biggl(\op{det}_{g}\Big(
f_u(u,0),\nabla_v f_{v}(u,0), \nabla_u f_{u}(u,0)\Big)\biggr)\in \{-1,0,1\},
\end{equation}
which is the numerator of the definition 
of limiting normal curvature $\kappa_\nu$
given in \cite[(2.2)]{MSUY}, 
since the unit normal vector field
is proportional to the vector $f_u(u,0)\times_g \nabla_v f_{v}(u,0)$
along the $u$-axis. 
The following assertion hold:

\begin{Prop}\label{prop:2136}
 For each germ of a generalized cuspidal edge in $\mc F^{GC}(M^3)$,
 the sign $\sigma^C_g(u)$ is 
locally constant and does not
 depend on the choice of an
 admissible coordinate system in the domain of definition.
 Moreover, the zero set of 
$\sigma_g^C$ coincides with
 that of the limiting normal curvature $\kappa_\nu$
 with respect to the metric $g$.
\end{Prop}

\begin{proof}
This is the generalization of the statement 
as \cite[Proposition 1.2]{SUY}.
Since, in the proof of \cite[Proposition 1.2]{SUY}, 
the fact that $f$ is a generalized cuspidal edge is
applied but not use the
fact that a germ of cuspidal edge. So we obtain the same conclusion.
\end{proof}

\section{The limit of the unit normal vector field 
at a generalized swallowtail singular point}
\label{App:A}

Let $f:U\to M^3$ be a generalized swallowtail belonging to 
the class $\mc F^{GS}(M^3)$.
Then the $u$-axis consists of the singular
points, and each $(u,0)$ ($u\ne 0$) is a generalized cuspidal edge singular point.
So, there exists a smooth vector field $\eta(u)$
along the $u$-axis giving the kernel of
the differential $df$ at $(u,0)$.
Then there exists a smooth vector field $\tilde \eta$ defined
on a sufficiently small neighborhood $U$ of $o$ such that
$\tilde \eta(u,0)=\eta(u)$ if $(u,0)\in U$.
The vector field $\tilde \eta$ is called an
{\it extended null vector field} (cf. \cite{SUY2}).
In \cite[(1.5)]{SUY2022},  
$\tilde\eta$ coincides with $\partial/\partial v$.
Since $(u,0)$ ($u\ne 0$) is a singular point of the first kind,
$$
\tilde \nu(u):=f_u(u,0)\times_g f_{\tilde\eta\tilde \eta}(u,0)
$$
can be taken as a normal direction of $f$
at $(u,0)$ ($u\ne 0$), where
$f_{\tilde \eta\tilde \eta}:=\nabla_{\tilde \eta} f_{\tilde \eta}$ 
($f_{\tilde \eta}:=df(\tilde \eta)$).
In this section, we prove the following:

\begin{Prop}
Let $f(u,v)$ be a generalized swallowtail belonging to 
the class $\mc F^{GS}(M^3)$.
Then, 
$$
(\nu_0:=)\lim_{u\to 0}\frac{\tilde \nu(u)}{|\tilde \nu(u)|}
$$
 exits and is
positive scalar multiplication of the vector
$
\nabla_u f_v(o)\times_g f_v(o)
$.
\end{Prop}

\begin{proof}
The existence of the limit
$\nu_0$ is obvious,
since $f$ is a frontal.
So, it is sufficient to prove the second assertion:
We can write
\begin{equation}\label{eq:tE}
\tilde \eta:=
\frac{\partial}{\partial u}
+\epsilon(u,v) \frac{\partial}{\partial v}
\end{equation}
on $U$, where $\epsilon$ is a 
certain smooth function on $U$.
Since 
$f_{\tilde \eta}(u,0)=0$, we can write
$$
v \psi(u,v)
=f_{\tilde \eta}(u,v)=f_u
+\epsilon(u,v) f_v.
$$
Differentiate it by $\nabla_v$, we have
\begin{equation}\label{eq:1909}
\psi(u,0)=\nabla_v f_u(u,0)+
\epsilon_v(u,0) f_v(u,0)+\epsilon(u,0)\nabla_v f_v(u,0).
\end{equation}
Since $o$ is a swallowtail,
$\eta(0)$ is proportional to $\partial/\partial u$,
and so $\epsilon(o)=0$.
So we have
\begin{equation}\label{eq:1913}
\psi(o)=\nabla_v f_u(o)+
\epsilon_v(o) f_v(o).
\end{equation}
On the other hand,
\eqref{eq:tE} implies
\begin{equation}\label{eq:FE1924}
f_\eta=f_u+\epsilon f_v,
\end{equation}
and we have
\begin{align*}
f_{\tilde \eta\tilde \eta}|_{v=0}
&=\nabla_{\tilde \eta} (v \psi)|_{v=0}
=dv(\eta)\psi(u,0) \\
&=
\left.\left(\frac{\partial v}{\partial u}
+\epsilon(u,v) \frac{\partial v}{\partial v}
\right)\right|_{v=0} \psi(u,0)=\epsilon(u,0) \psi(u,0).
\end{align*}
Since $f_\eta(u,0)=0$, we have
$
f_u(u,0)=-\epsilon(u,0)f_v(u,0).
$
So we have that
\begin{align*}
f_u(u,0)\times_g f_{\tilde\eta\tilde\eta}(u,0)
&=(-\epsilon(u,0)f_v(u,0))\times_g \epsilon(u,0) \psi(u,0) \\
&
=-\epsilon(u,0)^2 (f_v(u,0)\times_g \psi(u,0))
=\epsilon(u,0)^2  \psi(u,0)\times_g f_v(u,0).
\end{align*}
By
\eqref{eq:1913},
we have
$$
\psi(o)\times_g f_v(o)=
\nabla_v f_u(o)\times_g f_v(o),
$$
which proves the assertion.
\end{proof}

\begin{Remark}
In \cite[(4.14)]{MSUY}, 
the formula $f_v\times_g \psi=|\psi|\nu$ is given, where
the definition of $\psi$ is the same as ours, but
the condition that $\inner{f_{uv}(o)}{f_v(o)}=0$ 
is assumed, where $f_{uv}:=\nabla_{v} f_u$.
Since our $\nu(o)$ is positively proportional to 
$f_{uv}(o)\times_g f_v(o)$, 
the identity $f_v\times_g \psi=-|\psi|\nu$ holds
instead, that is,
the sign of $\nu$ in this paper is 
the opposite of that in \cite[(4.14)]{MSUY}.
\end{Remark}

As a consequence, we obtain the following:

\begin{Corollary}\label{cor:A1}
Let $f\in \mc F_{GS}(M^3)$ be a germ of generalized swallowtail.
There exists a unit normal vector field $\nu(u,v)$
satisfying the following properties;
\begin{enumerate}
\item the frame $f_u(u,0),\, f_{\tilde \eta\tilde \eta}(u,0),\,\nu(u,0)$
is positively oriented if $u\ne 0$, and
\item the frame $f_{uv}(o)\,,f_{v}(o),\,\nu(o)$
is positively oriented.
\end{enumerate}
\end{Corollary}

We next prove the following assertion:

\begin{Prop}\label{prop:A2}
In the above setting, 
$$
\inner{\nabla_uf_{u}(u,0)}{\tilde \nu(u,0)}=\epsilon(u,0)^2
\inner{\nabla_vf_{v}(u,0)}{\tilde \nu(u,0)}
\qquad (u\ne 0)
$$
holds.
\end{Prop}

\begin{proof}
Differentiate \eqref{eq:FE1924}, we have
$$
\nabla_uf_{u}(u,0)=-\epsilon_u(u,0)f_v(u,0)-\epsilon(u,0)
\nabla_u f_v(u,0).
$$
This with \eqref{eq:1909},
we have
\begin{align*}
\nabla_uf_{u}(u,0)&=
\epsilon(u,0)^2\nabla_vf_{v}(u,0)-
\epsilon_u(u,0)f_v(u,0) \\
&\phantom{aaaaaaaaaaaaaaa}
+\epsilon_v(u,0)f_v(u,0)-\epsilon(u,0)\psi(u,0),
\end{align*}
which implies the conclusion.
\end{proof}

\section{Nomalized half-arclength parameters for space-cusps}
\label{App:B}

In this section, we show the existence of the normalized arc-length
parametrization for a given space-cusp.

\begin{Definition}\label{def:1842}
Let $I$ be an open interval of $\R$ containing the origin $0$.
A $C^\infty$-map $\gamma:I\to \R^3$ is called  a {\it generalized
space-cusp} at $0$ if 
$\gamma'(0)=\mb 0$ but
$\gamma''(0)\ne \mb 0$.
\end{Definition}

We show the following:

\begin{Proposition}\label{prop:B1855}
For a given
generalized
space-cusp at $0$,
there exists a new parametrization of 
$\hat \gamma(u):=\gamma(t(u))$ such that
$d\hat\gamma/du=u\hat \xi(u)$ holds, where
$\hat \xi(u)$ is a unit vector field along $\hat \gamma$ in $\R^3$.
\end{Proposition}

\begin{proof}
Since $\gamma'(0)=\mb 0$ and
$\gamma''(0)\ne \mb 0$, we can write
$
\gamma'(t)=t \xi(t)
$,
where $\xi(t)$ is an $\R^3$-valued smooth function satisfying $\xi(0)\ne 
\mb 0$.
Then,
$$
\phi(t):=\int_0^t s |\xi(s)|ds
$$
satisfies $\phi(0)=\phi'(0)\ne 0$. So we can write
$
\phi(t)=t^2 \psi(t)\,\, (\psi(0)>0).
$
In particular, $u:=t\psi(t)=\op{sgn}(t)\sqrt{\phi(t)}$ 
gives a new parametrization of the curve $\gamma$.
By setting $\hat \gamma(u):=\gamma(t(u))$,
we have
\begin{align*}
\frac{d\hat \gamma(u)}{du}
&=\frac{\gamma'(t(u))}{du/dt}=
\frac{\gamma'(t(u))}{\op{sgn}(t)d\sqrt{\phi}/dt}\\
&=\frac{t \xi(t)}{\op{sgn}(t) t |\xi(t)|/\sqrt{\phi(t)}}
=\frac{t \xi(t)}{t |\xi(t)|/(t \psi(t))}
=u\hat \xi(u)\qquad (\xi(u):=\frac{\xi(t(u))}{|\xi(t(u))|}),
\end{align*}
proving the assertion.
\end{proof}

\begin{Remark}
In this situation, for $s\ne 0$,
\begin{align*}
\frac{\hat \gamma(\op{sgn}(s)\sqrt{2s})}{ds}
&=\hat \gamma'(\op{sgn}(s)\sqrt{2s})\op{sgn}(s)\frac{1}{\sqrt{2s}} \\
&=\op{sgn}(s)\sqrt{2s}\hat \xi(\op{sgn}(s)\sqrt{2s})\op{sgn}(s)\frac{1}{\sqrt{2s}}
=\hat \xi(\op{sgn}(s)\sqrt{2s})
\end{align*}
holds. Thus, the parameter $s$ gives the arc-length parameter of $\gamma$.
Since 
$u=\op{sgn}(s)\sqrt{2s}$, we have $u^2/2=s$, that is, 
$u$ gives the normalized arc-length parameter of $\gamma$ 
in the sense of \cite[Appendix B]{SUY}.
\end{Remark}

\section{An example of swallowtails with extended Gaussian curvature $1$}
\label{App1}

Let $F(r)$ be a $C^\infty$-function defined on an open
interval containing $r=1$ but not containing $r=0$, which is obtained by
the following initial value problem of the ordinary 
differential equation
\begin{equation}\label{eq:F}
F''(r)+\frac{F'(r)}{r}+\frac{\sinh 2F}{2}=0,\qquad
F(1)=0,\quad F'(1)=1.
\end{equation}
We set
$
\omega(u,v):=F(\sqrt{u^2+v^2}) 
$
and
$\triangle:=\frac{\partial^2}{\partial u^2}+\frac{\partial^2}{\partial v^2}$.
Then we have
$$
\omega_u(u,v)=\frac{u}{\sqrt{u^2+v^2}} F'(\sqrt{u^2+v^2}),\qquad
\omega_v(u,v)=\frac{v}{\sqrt{u^2+v^2}} F'(\sqrt{u^2+v^2})
$$
and
$$
\omega_{uu}(u,v)=
-\frac{u^2 F'\left(\sqrt{u^2+v^2}\right)}{\left(u^2+v^2\right)^{3/2}}
+\frac{F'\left(\sqrt{u^2+v^2}\right)}{\sqrt{u^2+v^2}}
+\frac{u^2 F''\left(\sqrt{u^2+v^2}\right)}{u^2+v^2}.
$$
By \eqref{eq:F}, 
$
\triangle \omega+\frac1{2}\sinh(2\omega)
=0
$
holds. We then consider the following two symmetric covariant tensors
\begin{equation}\label{eq:I-II910}
I:=e^{2\omega}(du^2+dv^2),\qquad
I\!I:=e^{\omega}(\cosh u du^2+\sinh u dv^2).
\end{equation}
Since they
satisfy the Gauss equation and the Codazzi equation, 
the fundamental theorem of surface theory (cf. \cite[Theorem~17.2]{UY}),
there exists an immersion $f:U\to \R^3$
defined on a neighborhood $U$
of $(0,1)\in \R^2$ whose first and second fundamental forms
are $I$ and $I\!I$, respectively,
and $f$ has constant mean curvature $1/2$.
By \eqref{eq:I-II910}, $f(u,v)$
is parametrized by
a curvature line coordinate system, and
$$
\lambda_1:=e^{-2\omega}\cosh u,\qquad
\lambda_2:=e^{-2\omega}\sinh u
$$
give two principal curvatures of $f(u,v)$.
At the point $(u,v)=(0,1)$, we have
$$
\lambda_1(o)=1,\qquad
\lambda_2(o)=0,
$$
and
$
(\lambda_1)_u=-2\omega_u \lambda_1+e^{-2\omega}\sinh u
$
vanihhes at $(u,v)=(0,1)$.
Moreover, 
$$
(\lambda_1)_{uu}(o)=-2\omega_{uu}(o) 
\lambda_1(o)+e^{-2\omega(o)}=
-2+1=-1 (\ne 0)
$$
and
$$
(\lambda_1)_v(o)=-2\omega_v(o) \lambda_1(o)=-2(\ne 0).
$$
By \cite[Theorem~4.3.3]{SUY2}, the parallel surface 
$h(u,v):=f(u,v)+\nu(u,v)$
associated with $f$ 
has a swallowtail singular point at $(u,v)=(0,1)$.
On the other hand, by \cite[Theorem~4.3.3]{SUY2},
$h$ is a surface of constant Gaussian curvature $1$.

\section{Corrections of the proof of Proposition 3.2 in \cite{SUY2022}}
\label{App4}

In the proof of Corollary E, we apply
\cite[Proposition~3.2]{SUY2022}.
However, in the proof of Proposition 3.2 in \cite{SUY2022},
there were some typographical 
errors. 
So we give here corrections 
along with the reasons for them.
We firstly remark that
\cite[(3.1)]{SUY2022} is correct,
but should be
 replaced by
\begin{equation}\label{eq:2193}
b^2+ca^2/k=1.
\end{equation}
In our setting, we have assumed that $k=\pm 1$,
so \eqref{eq:2193} coincides with
\cite[(3.1)]{SUY2022}.
In the proof of
\cite[Proposition~3.2]{SUY2022},
we have not explained why 
the identity \eqref{eq:2193} is induced:
In fact, if we set $d\tilde s^2=a^2ds^2$ and
$(\tilde A,\tilde B,\tilde C,\tilde D)=b(A,B,C,D)$,
then (A.8), (A.9), (A.10) and (A.11) induce the following
identities;
\begin{align}
\label{eq:747}
  &\tilde \rho \tilde B=\tilde \theta \tilde C, \\ 
\label{eq:748}
  &\tilde A\tilde D-\tilde B\tilde C+c\tilde \lambda=\alpha_v-\beta_u, \\ 
\label{eq:749}
  &\tilde B_u-\tilde A_v=\alpha \tilde D-\beta \tilde C, \\ 
\label{eq:750}
  &\tilde D_u-\tilde C_v=\beta \tilde A-\alpha \tilde B, 
\end{align}
where $\tilde \rho:=a \rho$, $\tilde \theta:=a \theta$,
$\tilde \lambda:=a^2\lambda$ and
the constant $c$ should be given as \eqref{eq:2193}
(in this computation, we used the fact that 
the pair $(\alpha,\beta)$ is invariant if
we change $ds^2$ by $a^2ds^2$).
These are the integrability condition for
the system of differential equation \cite[(A.7)]{SUY2022}
by setting $r^2=|a|$ and $e:=\op{sgn}(a)$.
So we obtain a new wave front $\tilde f$ (from a given wave front
$f$ in Euclidean $3$-space as in the statement of \cite[Proposition 3.2]{SUY2022})
whose first fundamental form is
$d\tilde s^2$ in $M^3(a)$.
We also remark that we do not use the assumption $k=\pm 1$
in the above arguments.

The paragraph just after \cite[(3.3)]{SUY2022} has several typos, so
we give here the corrections without assuming $\kappa=\pm 1$:

 If $c/k>0$, then we set
 $(a,b):=(\sqrt{k/c}\sin \theta,\cos \theta)$, where $0<\theta<\pi/2$.
 By \cite[(3.3)]{SUY2022}, 
the extrinsic curvature of $\tilde f$ satisfies
\begin{equation}\label{eq:A}
\tilde K_{ext}=c \cot^2 \theta.
\end{equation}
 On the other hand, if $c/k<0$, then we set 
 $(a,b):=(\sqrt{|k/c|}\sinh \theta,\cosh \theta)$,
 and then the extrinsic Gaussian curvature of $\tilde f$ satisfies
\begin{equation}\label{eq:B}
\tilde K_{ext}=-c\coth^2 \theta.
\end{equation}
So we obtain the first assertion of Proposition 3.2 of
\cite{SUY2022}.

\begin{acknowledgements}
The authors thank the reviewer
for valuable comments.
\end{acknowledgements}


\begin{thebibliography}{00}
 \bibitem{A} S.~Akamine, {\it Behavior of the Gaussian curvature of 
 time-like minimal surfaces with singularities},
 Hokkaido Math. J. {\bf 48} (2019), 537--568.

 \bibitem{F}
  T.~Fukui,
  {\it Local differential geometry of cuspidal edge and swallowtail},
Osaka J. Math. {\bf 57} (2020), 961--992.

\bibitem{FKPUY}
T. Fukui, R. Kinoshita, D. Pei, M. Umehara,
         	and H. Yu,
	{\it 
Cuspidal edges in Lorentz-Minkowski space},
	preprint (arXiv:2409.01603).
\bibitem{HNUY}
A. Honda, K. Naokawa, M. Umehara and K. Yamada, {\it Isometric deformations of wave
fronts at non-degenerate singular points}, Hiroshima Math. J. {\bf 50} (2020), 269--312.

\bibitem{HNSUY}
  A.~Honda, K.~Naokawa, K.~Saji, M.~Umehara and K.~Yamada,
  {\itshape 
Duality on generalized cuspidal edges preserving 
singular set images and 
first fundamental forms}, 
J. Singul.
{\bf 22} (2020), 59--91.

\bibitem{KN} S. Kobayashi and K. Nomizu,
Foundations of Differential Geometry Volume II,
John Wiley \& Sons, Inc. 1996.

\bibitem{KRSUY} M. Kokubu, W. Rossman, K. Saji, M. Umehara and K. Yamada, 
{\it Singularities of flat
fronts in hyperbolic space}, Pacific J. Math., {\bf 221} (2005), 303--351.


\bibitem{MSUY}
 L. F. Martins, K. Saji, M. Umehara and K. Yamada,
 \emph{Behavior of Gaussian curvature
	and mean curvature near non-degenerate
	singular points on wave fronts},
 Springer Proc.\ in Math.\ \& Stat. {\bf 154}
(2016), Springer, 247--282.

\bibitem{S} 
K. Saji, {\it Normal form of the swallowtail and its applications},
Internat. J. Math., {\bf 29} (2018), 1850046.
 

\bibitem{SUY}
K.~Saji, M.~Umehara and K.~Yamada,
{\it The geometry of fronts},
Ann. of Math., {\bf 169} (2009), 491--529.

\bibitem{SUY2}
K.~Saji, M.~Umehara and K.~Yamada,
{Differential Geometry of Curved and Surfaces 
with Singularities}, World Scientific (2021).

\bibitem{SUY2022}
K.~Saji, M.~Umehara and K.~Yamada,
{\it Deformations of cuspidal edges 
in a $3$-dimensional space form},
 Kodai Math. J. {\bf 47} (2024), 67--89. 

\bibitem{UY}
M.~Umehara and K.~Yamada,
{Differential Geometry of Curved and Surfaces}, 
World Scientific (2015).

\bibitem{UY80} 
M. Umehara and K. Yamada, 
{\it A deformation of tori with constant mean curvature in $\R^3$ to those 
in other space forms}, Trans. Amer. Math. Soc. {\bf 330} (1992), 845--857.
\end{thebibliography}
\end{document}